\documentclass[graybox]{svmult}
\usepackage{mathptmx}
\usepackage{helvet}
\usepackage{courier}
\usepackage{type1cm}
\usepackage{makeidx}
\usepackage{graphicx}
\usepackage{multicol}
\usepackage[bottom]{footmisc}

\makeindex

\input{tcilatex}

\begin{document}

\title*{ A Half-Discrete Hardy-Hilbert-Type Inequality with a Best Possible
Constant Factor Related to the Hurwitz Zeta Function }
\titlerunning{  A Half-Discrete Hardy-Hilbert-Type Inequality }
\author{Michael Th. Rassias and Bicheng Yang}
\institute{ M.~Th.~Rassias \at Department of Mathematics, ETH-Z\"{u}rich, R\"{a}mistrasse 101,
8092 Z\"{u}rich, Switzerland\\ \& Department of Mathematics, Princeton University, Fine Hall,
Washington Road, Princeton, NJ 08544-1000, USA\\ \email{michail.rassias@math.ethz.ch,
michailrassias@math.princeton.edu} \\ B. Yang \\ Department of Mathematics, Guangdong University of
 Education, Guangzhou,\\ Guangdong
 510303, P. R. China, \\
 e-mail: {bcyang@gdei.edu.cn\,\,\,\,bcyang818@163.com} }
%
%
\maketitle

\begin{abstract}
{Using methods of weight functions, techniques of real analysis as well as the Hermite-Hadamard inequality, a
half-discrete Hardy-Hilbert-type inequality with multi-parameters and a best possible constant
factor related to the Hurwitz zeta function and the Riemann zeta function is obtained.
Equivalent forms, normed operator expressions, their reverses and some particular
cases are also considered.\\}
\end{abstract}
\textbf{\ }

\section{Introduction}
If $p>1,\frac{1}{p}+\frac{1}{q}=1,f(x),g(y)\geq 0,f\in L^{p}(\mathbf{R}%
_{+}),g\in L^{q}(\mathbf{R}_{+}),$ $$||f||_{p}=(\int_{0}^{\infty
}f^{p}(x)dx)^{\frac{1}{p}}>0,$$
$||g||_{q}>0,$ then we have the following
Hardy-Hilbert's integral inequality (cf. \cite{HLP}):
\begin{equation}
\int_{0}^{\infty }\int_{0}^{\infty }\frac{f(x)g(y)}{x+y}dxdy<\frac{\pi }{%
\sin (\pi /p)}||f||_{p}||g||_{q},  \label{1.1}
\end{equation}%
where, the constant factor $\frac{\pi }{\sin (\pi /p)}$ is the best
possible. Assuming that $$a_{m},b_{n}\geq 0,a=\{a_{m}\}_{m=1}^{\infty }\in
l^{p},\ b=\{b_{n}\}_{n=1}^{\infty }\in l^{q},$$ $$||a||_{p}=(\sum_{m=1}^{\infty
}a_{m}^{p})^{\frac{1}{p}}>0,\ ||b||_{q}>0,$$ we have the following
Hardy-Hilbert's inequality with the same best constant $\frac{\pi }{\sin
(\pi /p)}$ (cf. \cite{HLP}):%
\begin{equation}
\sum_{m=1}^{\infty }\sum_{n=1}^{\infty }\frac{a_{m}b_{n}}{m+n}<\frac{\pi }{%
\sin (\pi /p)}||a||_{p}||b||_{q}.  \label{1.2}
\end{equation}%
Inequalities (\ref{1.1}) and (\ref{1.2}) are important in Analysis and its
applications (cf. \cite{HLP}, \cite{MPF}, \cite{Y1}, \cite{Y2}, \cite{Y4}).

If $\mu _{i},v _{j}>0(i,j\in \mathbf{N=\{}1,2,\cdots \mathbf{\}}),$%
\begin{equation}
U_{m}:=\sum_{i=1}^{m}\mu _{i},V_{n}:=\sum_{j=1}^{n}\nu _{j}(m,n\in \mathbf{N}%
),  \label{1.3}
\end{equation}%
then we have the following inequality (cf. \cite{HLP}, Theorem 321) :%
\begin{equation}
\sum_{m=1}^{\infty }\sum_{n=1}^{\infty }\frac{\mu _{m}^{1/q}\nu
_{n}^{1/p}a_{m}b_{n}}{U_{m}+V_{n}}<\frac{\pi }{\sin (\pi /p)}%
||a||_{p}||b||_{q}.  \label{1.4}
\end{equation}%
Replacing $\mu _{m}^{1/q}a_{m}$ and $v_{n}^{1/p}b_{n}$ by $a_{m}$
and $b_{n}$ in (\ref{1.4}), respectively, we obtain the following equivalent
form of (\ref{1.4}):%
\begin{equation}
\sum_{m=1}^{\infty }\sum_{n=1}^{\infty }\frac{a_{m}b_{n}}{U_{m}+V_{n}}<\frac{%
\pi }{\sin (\frac{\pi }{p})}\left( \sum_{m=1}^{\infty }\frac{a_{m}^{p}}{\mu
_{m}^{p-1}}\right) ^{\frac{1}{p}}\left( \sum_{n=1}^{\infty }\frac{b_{n}^{q}}{%
\nu _{n}^{q-1}}\right) ^{\frac{1}{q}}.  \label{1.5}
\end{equation}%
For $\mu _{i}=v_{j}=1(i,j\in \mathbf{N}),$ both (\ref{1.4}) and (\ref%
{1.5}) reduce to (\ref{1.2}). We call (\ref{1.4}) and (\ref{1.5}) as
Hardy-Hilbert-type inequalities.

\textbf{Note.} The authors  did not prove that (\ref{1.4}) is
valid with the best possible constant factor in  \cite{HLP}.

In 1998, by introducing an independent parameter $\lambda \in (0,1]$, Yang
\cite{Y3} gave an extension of (\ref{1.1}) with the kernel $1/(x+y)^{\lambda }$ for $p=q=2$. Optimizing the method used in \cite{Y3}, Yang
\cite{Y4} provided some extensions of (\ref{1.1}) and (\ref{1.2}) as follows:

If $\lambda _{1},\lambda _{2}\in \mathbf{R},\lambda _{1}+\lambda
_{2}=\lambda ,k_{\lambda }(x,y)$ is a non-negative homogeneous function of
degree $-\lambda ,$ with $$k(\lambda _{1})=\int_{0}^{\infty }k_{\lambda
}(t,1)t^{\lambda _{1}-1}dt\in \mathbf{R}_{+},$$ $\phi (x)=x^{p(1-\lambda
_{1})-1},\psi (x)=x^{q(1-\lambda _{2})-1},f(x),g(y)\geq 0,$
$$
f\in L_{p,\phi }(\mathbf{R}_{+})=\left\{ f;||f||_{p,\phi
}:=\{\int_{0}^{\infty }\phi (x)|f(x)|^{p}dx\}^{\frac{1}{p}}<\infty \right\} ,
$$
$g\in L_{q,\psi }(\mathbf{R}_{+}),||f||_{p,\phi },||g||_{q,\psi }>0,$ then we have
\begin{equation}
\int_{0}^{\infty }\int_{0}^{\infty }k_{\lambda }(x,y)f(x)g(y)dxdy<k(\lambda
_{1})||f||_{p,\phi }||g||_{q,\psi },  \label{1.6}
\end{equation}%
where, the constant factor $k(\lambda _{1})$ is the best possible.

Moreover,
if $k_{\lambda }(x,y)$ remains finite and $k_{\lambda }(x,y)x^{\lambda
_{1}-1}(k_{\lambda }(x,y)y^{\lambda _{2}-1})$ is decreasing with respect to $%
x>0\:(y>0),$ then for $a_{m,}b_{n}\geq 0,$%
$$
a\in l_{p,\phi }=\left\{ a;||a||_{p,\phi }:=(\sum_{n=1}^{\infty }\phi
(n)|a_{n}|^{p})^{\frac{1}{p}}<\infty \right\} ,
$$
$b=\{b_{n}\}_{n=1}^{\infty }\in l_{q,\psi },$ $||a||_{p,\phi },||b||_{q,\psi
}>0,$ we have
\begin{equation}
\sum_{m=1}^{\infty }\sum_{n=1}^{\infty }k_{\lambda
}(m,n)a_{m}b_{n}<k(\lambda _{1})||a||_{p,\phi }||b||_{q,\psi },  \label{1.7}
\end{equation}%
where, the constant factor $k(\lambda _{1})$ is still the best possible.

Clearly, for $$\lambda =1,k_{1}(x,y)=\frac{1}{x+y},\ \lambda _{1}=\frac{1}{q}%
,\ \lambda _{2}=\frac{1}{p},$$ inequality (\ref{1.6}) reduces to (\ref{1.1}),
while (\ref{1.7}) reduces to (\ref{1.2}). For 
$$0<\lambda _{1},\lambda
_{2}\leq 1,\ \lambda _{1}+\lambda _{2}=\lambda ,$$ we set%
$$
k_{\lambda }(x,y)=\frac{1}{(x+y)^{\lambda }}\ ((x,y)\in \mathbf{R}_{+}^{2}).
$$
Then by (\ref{1.7}), we have
\begin{equation}
\sum_{m=1}^{\infty }\sum_{n=1}^{\infty }\frac{a_{m}b_{n}}{(m+n)^{\lambda }}%
<B(\lambda _{1},\lambda _{2})||a||_{p,\phi }||b||_{q,\psi },  \label{1.8}
\end{equation}%
where, the constant $B(\lambda _{1},\lambda _{2})$ is the best possible, and
$$
B\left( u,v\right) =\int_{0}^{\infty }\frac{1}{(1+t)^{u+v}}%
t^{u-1}dt^{{}}(u,v>0)
$$
is the beta function.

In 2015, subject to further conditions, Yang \cite{Y4a}
proved an extension of (\ref{1.8}) and (\ref{1.5}) as follows:
\begin{eqnarray}
&&\sum_{m=1}^{\infty }\sum_{n=1}^{\infty }\frac{a_{m}b_{n}}{%
(U_{m}+V_{n})^{\lambda }}   \\
&<&B(\lambda _{1},\lambda _{2})\left( \sum_{m=1}^{\infty }\frac{%
U_{m}^{p(1-\lambda _{1})-1}a_{m}^{p}}{\mu _{m}^{p-1}}\right) ^{\frac{1}{p}%
}\left( \sum_{n=1}^{\infty }\frac{V_{n}^{q(1-\lambda _{2})-1}b_{n}^{q}}{\nu
_{n}^{q-1}}\right) ^{\frac{1}{q}},  \label{1.9}
\end{eqnarray}%
where, the constant $B(\lambda _{1},\lambda _{2})$ is still the best
possible.

Further results including some multidimensional Hilbert-type inequalities can be found in \cite{YMP}- \cite{Y6}.

On the topic of half-discrete Hilbert-type inequalities with 
non-homogeneous kernels, Hardy et al. provided a few results in Theorem 351
of \cite{HLP}. But, they did not prove that the constant factors are the
best possible. However, Yang \cite{Y5} presented a result with the kernel $1/
(1+nx)^{\lambda }$ by introducing a variable and proved that the constant
factor is the best possible. In 2011, Yang \cite{122} gave the following
half-discrete Hardy-Hilbert's inequality with the best possible constant
factor $B\left( \lambda _{1},\lambda _{2}\right) $:%
\begin{equation}
\int_{0}^{\infty }f\left( x\right) \left[ \sum_{n=1}^{\infty }\frac{a_{n}}{%
\left( x+n\right) ^{\lambda }}\right] dx<B\left( \lambda _{1},\lambda
_{2}\right) ||f||_{p,\phi }||a||_{q,\psi },  \label{1.10}
\end{equation}%
where, $\lambda _{1}>0$, $0<\lambda _{2}\leq 1$, $\lambda _{1}+\lambda
_{2}=\lambda .$ Zhong et al. (\cite{140}--\cite{144}) investigated
several half-discrete Hilbert-type inequalities with particular kernels.

Using methods of weight functions and techniques of discrete and
integral Hilbert-type inequalities with some additional conditions on the
kernel, a half-discrete Hilbert-type inequality with a general homogeneous
kernel of degree $-\lambda \in \mathbf{R}$ and a best constant factor $%
k\left( \lambda _{1}\right) $ is obtained as follows:%
\begin{equation}
\int_{0}^{\infty }f(x)\sum_{n=1}^{\infty }k_{\lambda }(x,n)a_{n}dx<k(\lambda
_{1})||f||_{p,\phi }||a||_{q,\psi },  \label{1.11}
\end{equation}%
which is an extension of (\ref{1.10}) ( cf. Yang and Chen \cite{125}).
Additionally, a half-discrete Hilbert-type inequality with a general
non-homogeneous kernel and a best constant factor is given by Yang \cite{YB1}%
. The reader is referred to the three books  of Yang et al. \cite{YB8}, \cite{YB9} and
\cite{YB7}, where half-discrete Hilbert-type
inequalities and their operator expressions are extensively treated.

In this paper, using methods of weight functions, techniques of real analysis as well as the Hermite-Hadamard inequality, a
half-discrete Hardy-Hilbert-type inequality with multi-parameters and a best possible constant
factor related to the Hurwitz zeta function and the Riemann zeta function is obtained, which is an extension of (\ref{1.11}) for $\lambda =0$ in a particular
kernel. Equivalent forms, normed operator expressions, their reverses and some particular
cases are also considered.

\section{An Example and Some Lemmas}

In the following, we assume that $\mu _{i},\nu _{j}>0\: (i,j\in \mathbf{N}%
),U_{m}$ and $V_{n}$ are defined by (\ref{1.3}), $$\widetilde{V}_{n}:=V_{n}-%
\widetilde{\nu }_{n}(\widetilde{\nu }_{n}\in \lbrack 0,\frac{\nu _{n}}{2}%
])(n\in \textbf{N}),$$ $\mu (t)$ is a positive continuous function in $\mathbf{R}_{+}=(0,\infty )$%
,
$$
U(x):=\int_{0}^{x}\mu (t)dt<\infty ^{{}}(x\in \lbrack 0,\infty )),
$$
$\nu (t):=\nu _{n},t\in (n-\frac{1}{2},n+\frac{1}{2}](n\in \mathbf{N}),$ and
$$
V(y):=\int_{\frac{1}{2}}^{y}\nu (t)dt^{{}}(y\in \lbrack \frac{1}{2},\infty
)),
$$
$p\neq 0,1,$ $\frac{1}{p}+\frac{1}{q}=1,\delta \in \{-1,1\},$ $%
f(x),a_{n}\geq 0(x\in \mathbf{R}_{+},n\in \mathbf{N}),$ $$||f||_{p,\Phi
_{\delta }}=(\int_{0}^{\infty }\Phi _{\delta }(x)f^{p}(x)dx)^{\frac{1}{p}},$$
$||a||_{q,\widetilde{\Psi }}=(\sum_{n=1}^{\infty }\widetilde{\Psi }%
(n)b_{n}^{q})^{\frac{1}{q}},$ where,
$$
\Phi _{\delta }(x):=\frac{U^{p(1-\delta \sigma )-1}(x)}{\mu ^{p-1}(x)},%
\widetilde{\Psi }(n):=\frac{\widetilde{V}_{n}^{q(1-\sigma )-1}}{\nu
_{n}^{q-1}}(x\in \mathbf{R}_{+},n\in \mathbf{N}).
$$

\begin{example}
 For $0<\gamma <\sigma,0\leq \alpha \leq
\rho\ (\rho >0),$ $$\csc h(u):=\frac{2}{e^{u}-e^{-u}}\ (u>0)$$ is the
hyperbolic cosecant function (cf. \cite{ZYQ}). \ We set%
$$
h(t)=\frac{\csc h(\rho t^{\gamma })}{e^{\alpha t^{\gamma }}}\ (t\in \mathbf{R}%
_{+}).
$$
\end{example}

(i) Setting $u=\rho t^{\gamma},$  we find%
\begin{eqnarray*}
k(\sigma ) &:=&\int_{0}^{\infty }\frac{\csc h(\rho t^{\gamma })}{e^{\alpha
t^{\gamma }}}t^{\sigma -1}dt \\
&=&\frac{1}{\gamma \rho ^{\sigma /\gamma }}%
\int_{0}^{\infty }\frac{\csc h(u)}{e^{\frac{\alpha }{\rho }u}}u^{\frac{%
\sigma }{\gamma }-1}du \\
&=&\frac{2}{\gamma \rho ^{\sigma /\gamma }}\int_{0}^{\infty }\frac{e^{-\frac{%
\alpha }{\rho }u}u^{\frac{\sigma }{\gamma }-1}}{e^{u}-e^{-u}}du \\
&=&\frac{2}{\gamma \rho ^{\sigma /\gamma }}\int_{0}^{\infty }\frac{e^{-(%
\frac{\alpha }{\rho }+1)u}u^{\frac{\sigma }{\gamma }-1}}{1-e^{-2u}}du \\
&=&\frac{2}{\gamma \rho ^{\sigma /\gamma }}\int_{0}^{\infty
}\sum_{k=0}^{\infty }e^{-(2k+\frac{\alpha }{\rho }+1)u}u^{\frac{\sigma }{%
\gamma }-1}du.
\end{eqnarray*}%
By the Lebesgue term by term integration theorem (cf. \cite{ZYQ}), setting $v=\left(2k+\frac{\alpha}{\rho}+1\right)u$, we have%
\begin{eqnarray}
k(\sigma ) &=&\int_{0}^{\infty }\frac{\csc h(\rho t^{\gamma })}{e^{\alpha
t^{\gamma }}}t^{\sigma -1}dt  \nonumber \\
&=&\frac{2}{\gamma \rho ^{\sigma /\gamma }}\sum_{k=0}^{\infty
}\int_{0}^{\infty }e^{-(2k+\frac{\alpha }{\rho }+1)u}u^{\frac{\sigma }{%
\gamma }-1}du  \nonumber \\
&=&\frac{2}{\gamma \rho ^{\sigma
/\gamma }}\sum_{k=0}^{\infty }\frac{1}{(2k+\frac{\alpha }{\rho }+1)^{\sigma
/\gamma }}\int_{0}^{\infty }e^{-v}v^{\frac{\sigma }{\gamma }-1}dv  \nonumber \\
&=&\frac{2\Gamma (\frac{\sigma }{\gamma })}{\gamma (2\rho )^{\sigma /\gamma }%
}\sum_{k=0}^{\infty }\frac{1}{(k+\frac{\alpha +\rho }{2\rho })^{\sigma
/\gamma }}  \nonumber \\
&=&\frac{2\Gamma (\frac{\sigma }{\gamma })}{\gamma (2\rho )^{\sigma /\gamma }%
}\zeta (\frac{\sigma }{\gamma },\frac{\alpha +\rho }{2\rho })\in \mathbf{R}%
_{+},  \label{2.1}
\end{eqnarray}%
where, $$\zeta (s,a):=\sum_{k=0}^{\infty }\frac{1}{(k+a)^{s}}\ (s>1;0<a\leq 1)$$
is the Hurwitz zeta function, $\zeta (s)=\zeta (s,1)$ is the Riemann
zeta function, and%
$$
\Gamma (y):=\int_{0}^{\infty }e^{-v}v^{y-1}dv^{{}}\ (y>0)
$$
is the Gamma function (cf. \cite{Y4b}).

In particular, (1) for $\alpha =\rho >0,$we have $h(t)=\frac{\csc h(\rho
t^{\gamma })}{e^{\rho t^{\gamma }}}$ and $k(\sigma )=\frac{2\Gamma (\frac{%
\sigma }{\gamma })\zeta (\frac{\sigma }{\gamma })}{\gamma (2\rho )^{\sigma
/\gamma }}.$ In this case, for $\gamma =\frac{\sigma }{2},$ we have $h(t)=%
\frac{\csc h(\rho \sqrt{t^{\sigma }})}{e^{\rho \sqrt{t^{\sigma }}}}$ and $%
k(\sigma )=\frac{\pi ^{2}}{6\sigma \rho ^{2}};$ (2) for $\alpha =0,$ we have
$h(t)=\csc h(\rho t^{\gamma })$ and $\frac{2\Gamma (\frac{\sigma }{\gamma })%
}{\gamma (2\rho )^{\sigma /\gamma }}\zeta (\frac{\sigma }{\gamma },\frac{1}{2%
}).$ In this case, for $\gamma =\frac{\sigma }{2},$ we find $h(t)=\csc
h(\rho \sqrt{t^{\sigma }})$ and $k(\sigma )=\frac{\pi ^{2}}{2\sigma \rho ^{2}%
}.$

(ii) We obtain for $u>0,\frac{1}{e^{u}-e^{-u}}>0,$%
\begin{eqnarray*}
\frac{d}{du}(\frac{1}{e^{u}-e^{-u}}) &=&-\frac{e^{u}+e^{-u}}{%
(e^{u}-e^{-u})^{2}}<0, \\
\frac{d^{2}}{du^{2}}(\frac{1}{e^{u}-e^{-u}}) &=&\frac{%
2(e^{u}+e^{-u})^{2}-(e^{u}-e^{-u})^{2}}{(e^{u}-e^{-u})^{3}}>0.
\end{eqnarray*}

If $g(u)>0,\:g^{\prime }(u)<0,\:g^{\prime \prime }(u)>0,$ then for $0<\gamma
\leq 1,$%
\begin{eqnarray*}
g(\rho t^{\gamma }) &>&0,\frac{d}{dt}g(\rho t^{\gamma })=\rho \gamma
t^{\gamma -1}g^{\prime }(\rho t^{\gamma })<0, \\
\frac{d^{2}}{dt^{2}}g(\rho t^{\gamma }) &=&\rho \gamma (\gamma -1)t^{\gamma
-2}g^{\prime }(\rho t^{\gamma })+\rho ^{2}\gamma ^{2}t^{2\gamma -2}g^{\prime
\prime }(\rho t^{\gamma })>0;
\end{eqnarray*}%
for $y\in (n-\frac{1}{2},n+\frac{1}{2}),g(V(y))>0,$%
\begin{eqnarray*}
\frac{d}{dy}g(V(y)) &=&g^{\prime }(V(y))\nu _{n}<0, \\
\frac{d^{2}}{dy^{2}}g(V(y)) &=&g^{\prime \prime }(V(y))\nu _{n}^{2}>0(n\in
\mathbf{N}).
\end{eqnarray*}

If $g_{i}(u)>0,g_{i}^{\prime }(u)<0,g_{i}^{\prime \prime }(u)>0(i=1,2),$
then
\begin{eqnarray*}
g_{1}(u)g_{2}(u) &>&0, \\
(g_{1}(u)g_{2}(u))^{\prime } &=&g_{1}^{\prime
}(u)g_{2}(u)+g_{1}(u)g_{2}^{\prime }(u)<0, \\
(g_{1}(u)g_{2}(u))^{\prime \prime } &=&g_{1}^{\prime \prime
}(u)g_{2}(u)+2g_{1}^{\prime }(u)g_{2}^{\prime }(u)+g_{1}(u)g_{2}^{\prime
\prime }(u)>0(u>0).
\end{eqnarray*}

(iii) Therefore, for $0<\gamma<\sigma \leq 1, 0\leq \alpha \leq
\rho (\rho >0),$ we have $k(\sigma)\in \textbf{R}_{+},$ with $ h(t)>0,h^{\prime
}(t)<0,h^{\prime \prime }(t)>0,$ and then for $c>0,y\in(n-\frac{1}{2},n+\frac{1}{2})(n\in \textbf{N}),$ it follows that
\begin{eqnarray*}
&&h(cV(y))V^{\sigma -1}(y) >0,\\&& \frac{d}{dy}h(cV(y))V^{\sigma -1}(y)<0, \\
&&\frac{d^{2}}{dy^{2}}h(cV(y))V^{\sigma -1}(y) >0.
\end{eqnarray*}

\begin{lemma}
 If $g(t)(>0)$ is decreasing in $\mathbf{R}_{+}$ and
strictly decreasing in $[n_{0},\infty )$ where $n_{0}\in \mathbf{N},$ satisfying $%
\int_{0}^{\infty }g(t)dt\in \mathbf{R}_{+},$ then we have%
\begin{equation}
\int_{1}^{\infty }g(t)dt<\sum_{n=1}^{\infty }g(n)<\int_{0}^{\infty }g(t)dt.
\label{2.2}
\end{equation}
\end{lemma}

\begin{proof}
Since we have
\begin{eqnarray*}
\int_{n}^{n+1}g(t)dt &\leq &g(n)\leq \int_{n-1}^{n}g(t)dt(n=1,\cdots ,n_{0}),
\\
\int_{n_{0}+1}^{n_{0}+2}g(t)dt &<&g(n_{0}+1)<\int_{n_{0}}^{n_{0}+1}g(t)dt,
\end{eqnarray*}%
then it follows that%
$$
0<\int_{1}^{n_{0}+2}g(t)dt<\sum_{n=1}^{n_{0}+1}g(n)<\sum_{n=1}^{n_{0}+1}%
\int_{n-1}^{n}g(t)dt=\int_{0}^{n_{0}+1}g(t)dt<\infty .
$$
Similarly, we still have
$$
0<\int_{n_{0}+2}^{\infty }g(t)dt\leq \sum_{n=n_{0}+2}^{\infty }g(n)\leq
\int_{n_{0}+1}^{\infty }g(t)dt<\infty .
$$
Hence, (\ref{2.2}) follows and therefore the lemma is proved.
\end{proof}

\begin{lemma}
  If $0\leq \alpha \leq \rho (\rho >0),0<\gamma
<\sigma \leq 1,$ define the following weight coefficients:%
\begin{eqnarray}
\omega _{\delta }(\sigma ,x) &:&=\sum_{n=1}^{\infty }\frac{\csc h(\rho
(U^{\delta }(x)\widetilde{V}_{n})^{\gamma })}{e^{\alpha (U^{\delta }(x)%
\widetilde{V}_{n})^{\gamma }}}\frac{U^{\delta \sigma }(x)\nu _{n}}{%
\widetilde{V}_{n}^{1-\sigma }},x\in \mathbf{R}_{+},  \label{2.3} \\
\varpi _{\delta }(\sigma ,n) &:&=\int_{0}^{\infty }\frac{\csc h(\rho
(U^{\delta }(x)\widetilde{V}_{n})^{\gamma })}{e^{\alpha (U^{\delta }(x)%
\widetilde{V}_{n})^{\gamma }}}\frac{\widetilde{V}_{n}^{\sigma }\mu (x)}{%
U^{1-\delta \sigma }(x)}dx,n\in \mathbf{N}.  \label{2.4}
\end{eqnarray}%
Then, we have the following inequalities:
\begin{eqnarray}
\omega _{\delta }(\sigma ,x) &<&k(\sigma )(x\in \mathbf{R}_{+}),  \label{2.5}
\\
\varpi _{\delta }(\sigma ,n) &\leq &k(\sigma )(n\in \mathbf{N}),  \label{2.6}
\end{eqnarray}%
where, $k(\sigma )$ is given by (\ref{2.1}).
\end{lemma}

\begin{proof}
  Since we find
\begin{eqnarray*}
\widetilde{V}_{n} &=&V_{n}-\widetilde{\nu }_{n}\geq V_{n}-\frac{\nu _{n}}{2}
\\
&=&\int_{\frac{1}{2}}^{n+\frac{1}{2}}\nu (t)dt-\int_{n}^{n+\frac{1}{2}}\nu
(t)dt=\int_{\frac{1}{2}}^{n}\nu (t)dt=V(n),
\end{eqnarray*}%
and for $t\in (n-\frac{1}{2},n+\frac{1}{2}],V^{\prime }(t)=\nu _{n},$ then
by Example 1(iii) and Hermite-Hadamard's inequality (cf. \cite{K1}), we have
\begin{eqnarray*}
&&\frac{\csc h(\rho (U^{\delta }(x)\widetilde{V}_{n})^{\gamma })}{e^{\alpha
(U^{\delta }(x)\widetilde{V}_{n})^{\gamma }}}\frac{\nu _{n}}{\widetilde{V}%
_{n}^{1-\sigma }} \\
&\leq &\frac{\csc h(\rho (U^{\delta }(x)V(n))^{\gamma })}{e^{\alpha
(U^{\delta }(x)V(n))^{\gamma }}}\frac{\nu _{n}}{V^{1-\sigma }(n)} \\
&<&\int_{n-\frac{1}{2}}^{n+\frac{1}{2}}\frac{\csc h(\rho (U^{\delta
}(x)V(t))^{\gamma })}{e^{\alpha (U^{\delta }(x)V(t))^{\gamma }}}\frac{%
V^{\prime }(t)}{V^{1-\sigma }(t)}dt,
\end{eqnarray*}%
\begin{eqnarray*}
\omega _{\delta }(\sigma ,x) &<&\sum_{n=1}^{\infty }\int_{n-\frac{1}{2}}^{n+%
\frac{1}{2}}\frac{\csc h(\rho (U^{\delta }(x)V(t))^{\gamma })}{e^{\alpha
(U^{\delta }(x)V(t))^{\gamma }}}\frac{U^{\delta \sigma }(x)V^{\prime }(t)}{%
V^{1-\sigma }(t)}dt \\
&=&\int_{\frac{1}{2}}^{\infty }\frac{\csc h(\rho (U^{\delta
}(x)V(t))^{\gamma })}{e^{\alpha (U^{\delta }(x)V(t))^{\gamma }}}\frac{%
U^{\delta \sigma }(x)V^{\prime }(t)}{V^{1-\sigma }(t)}dt.
\end{eqnarray*}%
Setting $u=U^{\delta }(x)V(t),$ by (\ref{2.1}), we find%
\begin{eqnarray*}
\omega _{\delta }(\sigma ,x) &<&\int_{0}^{U^{\delta }(x)V(\infty )}\frac{%
\csc h(\rho u^{\gamma })}{e^{\alpha u^{\gamma }}}\frac{U^{\delta \sigma
}(x)U^{-\delta }(x)}{(uU^{-\delta }(x))^{1-\sigma }}du \\
&\leq &\int_{0}^{\infty }\frac{\csc h(\rho u^{\gamma })}{e^{\alpha u^{\gamma
}}}u^{\sigma -1}du=k(\sigma ).
\end{eqnarray*}%
Hence, (\ref{2.5}) follows.

Setting $u=\widetilde{V}_{n}U^{\delta }(x)$ in (\ref{2.4}), we find $%
du=\delta \widetilde{V}_{n}U^{\delta -1}(x)\mu (x)dx$ and
\begin{eqnarray*}
\varpi _{\delta }(\sigma ,n) &=&\frac{1}{\delta }\int_{\widetilde{V}%
_{n}U^{\delta }(0)}^{\widetilde{V}_{n}U^{\delta }(\infty )}\frac{\csc h(\rho
u^{\gamma })}{e^{\alpha u^{\gamma }}}\frac{\widetilde{V}_{n}^{\sigma }%
\widetilde{V}_{n}^{-1}(\widetilde{V}_{n}^{-1}u)^{\frac{1}{\delta }-1}}{(%
\widetilde{V}_{n}^{-1}u)^{\frac{1}{\delta }-\sigma }}du \\
&=&\frac{1}{\delta }\int_{\widetilde{V}_{n}U^{\delta }(0)}^{\widetilde{V}%
_{n}U^{\delta }(\infty )}\frac{\csc h(\rho u^{\gamma })}{e^{\alpha u^{\gamma
}}}u^{\sigma -1}du.
\end{eqnarray*}%
If $\delta =1,$ then%
\begin{eqnarray*}
\varpi _{1}(\sigma ,n) &=&\int_{0}^{\widetilde{V}_{n}U(\infty )}\frac{\csc
h(\rho u^{\gamma })}{e^{\alpha u^{\gamma }}}u^{\sigma -1}du \\
&\leq &\int_{0}^{\infty }\frac{\csc h(\rho u^{\gamma })}{e^{\alpha u^{\gamma
}}}u^{\sigma -1}du.
\end{eqnarray*}%
If $\delta =-1,$ then%
\begin{eqnarray*}
\varpi _{-1}(\sigma ,n) &=&-\int_{\infty }^{\widetilde{V}_{n}U^{-1}(\infty )}%
\frac{\csc h(\rho u^{\gamma })}{e^{\alpha u^{\gamma }}}u^{\sigma -1}du \\
&\leq &\int_{0}^{\infty }\frac{\csc h(\rho u^{\gamma })}{e^{\alpha u^{\gamma
}}}u^{\sigma -1}du.
\end{eqnarray*}%
Then by (\ref{2.1}), we have (\ref{2.6}).  The lemma is proved.
\end{proof}

\begin{remark}
  We do not need the constraint $\sigma \leq 1$ to
obtain (\ref{2.6}). If $U(\infty )=\infty ,$ then we have
\begin{equation}
\varpi _{\delta }(\sigma ,n)=k(\sigma )(n\in \mathbf{N}).  \label{2.7}
\end{equation}%
For example, if we set $\mu (t)=\frac{1}{(1+t)^{\beta }}(t>0;0\leq \beta \leq
1),$ then for $x\geq 0,$ we find
\begin{eqnarray*}
U(x) &=&\int_{0}^{x}\frac{1}{(1+t)^{\beta }}dt \\
&=&\left\{
\begin{array}{c}
\frac{(1+x)^{1-\beta }-1}{1-\beta },0\leq \beta <1 \\
\ln (1+x),\beta =1%
\end{array}%
\right. <\infty ,
\end{eqnarray*}%
and $$U(\infty )=\int_{0}^{\infty }\frac{1}{(1+t)^{\beta }}dt=\infty .$$
\end{remark}

\begin{lemma}
  If $0\leq \alpha \leq \rho\ (\rho >0),\ 0<\gamma <\sigma \leq
1,$ there exists $n_{0}\in \mathbf{N},$ such that $\nu _{n}\geq \nu _{n+1}$ $%
(n\in \mathbf{\{}n_{0},n_{0}+1,\cdots \}),$ and $V(\infty )=\infty ,$ then,\\
(i) for $x\in \mathbf{R}_{+}\mathbf{,}$ we have
\begin{equation}
k(\sigma )(1-\theta _{\delta }(\sigma ,x))<\omega _{\delta }(\sigma ,x),
\label{2.8}
\end{equation}%
where, $\theta _{\delta }(\sigma ,x)=O((U(x))^{\delta (\sigma -\gamma )})\in
(0,1);$\\ 
(ii) for any $b>0,$ we have%
\begin{equation}
\sum_{n=1}^{\infty }\frac{\nu _{n}}{\widetilde{V}_{n}^{1+b}}=\frac{1}{b}%
\left( \frac{1}{V_{n_{0}}^{b}}+bO(1)\right) .  \label{2.9}
\end{equation}
\end{lemma}

\begin{proof}
 Since $v_{n}\geq v_{n+1}(n\geq n_{0}),$ and
$$
\widetilde{V}_{n}=V_{n}-\widetilde{\nu }_{n}\leq V_{n}=\int_{\frac{1}{2}}^{n+%
\frac{1}{2}}\nu (t)dt=V(n+\frac{1}{2}),
$$
by Example 1(iii), we have%
\begin{eqnarray*}
\omega _{\delta }(\sigma ,x) &=&\sum_{n=1}^{\infty }\frac{\csc h(\rho
(U^{\delta }(x)\widetilde{V}_{n})^{\gamma })}{e^{\alpha (U^{\delta }(x)%
\widetilde{V}_{n})^{\gamma }}}\frac{U^{\delta \sigma }(x)\nu _{n}}{%
\widetilde{V}_{n}^{1-\sigma }} \\
&\geq &\sum_{n=n_{0}}^{\infty }\int_{n+\frac{1}{2}}^{n+\frac{3}{2}}\frac{%
\csc h(\rho (U^{\delta }(x)V(n+\frac{1}{2}))^{\gamma })}{e^{\alpha
(U^{\delta }(x)V(n+\frac{1}{2}))^{\gamma }}}\frac{U^{\delta \sigma }(x)\nu
_{n+1}dt}{(V(n+\frac{1}{2}))^{1-\sigma }} \\
&>&\sum_{n=n_{0}}^{\infty }\int_{n+\frac{1}{2}}^{n+\frac{3}{2}}\frac{\csc
h(\rho (U^{\delta }(x)V(t))^{\gamma })}{e^{\alpha (U^{\delta
}(x)V(t))^{\gamma }}}\frac{U^{\delta \sigma }(x)V^{\prime }(t)}{%
(V(t))^{1-\sigma }}dt \\
&=&\int_{n_{0}+\frac{1}{2}}^{\infty }\frac{\csc h(\rho (U^{\delta
}(x)V(t))^{\gamma })}{e^{\alpha (U^{\delta }(x)V(t))^{\gamma }}}\frac{%
U^{\delta \sigma }(x)V^{\prime }(t)}{(V(t))^{1-\sigma }}dt.
\end{eqnarray*}%
Setting $u=U^{\delta }(x)V(t),$ in view of $V(\infty )=\infty ,$ by (\ref%
{2.1}), we find%
\begin{eqnarray*}
\omega _{\delta }(\sigma ,x) &>&\int_{U^{\delta }(x)V_{n_{0}}}^{\infty }%
\frac{\csc h(\rho u^{\gamma })}{e^{\alpha u^{\gamma }}}u^{\sigma -1}du \\
&=&k(\sigma )-\int_{0}^{U^{\delta }(x)V_{n_{0}}}\frac{\csc h(\rho u^{\gamma
})}{e^{\alpha u^{\gamma }}}u^{\sigma -1}du \\
&=&k(\sigma )(1-\theta _{\delta }(\sigma ,x)), \\
\theta _{\delta }(\sigma ,x) &:&=\frac{1}{k(\sigma )}\int_{0}^{U^{\delta
}(x)V_{n_{0}}}\frac{\csc h(\rho u^{\gamma })}{e^{\alpha u^{\gamma }}}%
u^{\sigma -1}du\in (0,1).
\end{eqnarray*}

Since $F(u)=\frac{u^{\gamma }\csc h(\rho u^{\gamma })}{e^{\alpha u^{\gamma }}%
}$ is continuous in $(0,\infty ),$ satisfying $$F(u)\rightarrow \frac{1}{\rho }%
(u\rightarrow 0^{+}),F(u)\rightarrow 0(u\rightarrow \infty ),$$ there exists
a constant $L>0,$ such that $F(u)\leq L,$ namely,
$$
\frac{\csc h(\rho u^{\gamma })}{e^{\alpha u^{\gamma }}}\leq Lu^{-\gamma
}(u\in (0,\infty )).
$$
Hence we find%
\begin{eqnarray*}
0 &<&\theta _{\delta }(\sigma ,x)\leq \frac{L}{k(\sigma )}%
\int_{0}^{U^{\delta }(x)V_{n_{0}}}u^{\sigma -\gamma -1}du \\
&=&\frac{L(U^{\delta }(x)V_{n_{0}})^{\sigma -\gamma }}{k(\sigma )(\sigma
-\gamma )},
\end{eqnarray*}%
and then (\ref{2.8}) follows.

For $b>0,$ we find%
\begin{eqnarray*}
\sum_{n=1}^{\infty }\frac{\nu _{n}}{\widetilde{V}_{n}^{1+b}} &\leq
&\sum_{n=1}^{n_{0}}\frac{\nu _{n}}{\widetilde{V}_{n}^{1+b}}%
+\sum_{n=n_{0}+1}^{\infty }\frac{\nu _{n}}{V^{1+b}(n)} \\
&<&\sum_{n=1}^{n_{0}}\frac{\nu _{n}}{\widetilde{V}_{n}^{1+b}}%
+\sum_{n=n_{0}+1}^{\infty }\int_{n-\frac{1}{2}}^{n+\frac{1}{2}}\frac{%
V^{\prime }(x)}{V^{1+b}(x)}dx \\
&=&\sum_{n=1}^{n_{0}}\frac{\nu _{n}}{\widetilde{V}_{n}^{1+b}}+\int_{n_{0}+%
\frac{1}{2}}^{\infty }\frac{dV(x)}{V^{1+b}(x)} \\
&=&\sum_{n=1}^{n_{0}}\frac{\nu _{n}}{\widetilde{V}_{n}^{1+b}}+\frac{1}{%
bV^{b}(n_{0}+\frac{1}{2})} \\
&=&\frac{1}{b}\left( \frac{1}{V_{n_{0}}^{b}}+b\sum_{n=1}^{n_{0}}\frac{\nu
_{n}}{\widetilde{V}_{n}^{1+b}}\right) ,
\end{eqnarray*}%
\begin{eqnarray*}
\sum_{n=1}^{\infty }\frac{\nu _{n}}{\widetilde{V}_{n}^{1+b}} &\geq
&\sum_{n=n_{0}}^{\infty }\int_{n+\frac{1}{2}}^{n+\frac{3}{2}}\frac{\nu _{n+1}%
}{V^{1+b}(n+\frac{1}{2})}dx \\
&>&\sum_{n=n_{0}}^{\infty }\int_{n+\frac{1}{2}}^{n+\frac{3}{2}}\frac{%
V^{\prime }(x)}{V^{1+b}(x)}dx=\int_{n_{0}+\frac{1}{2}}^{\infty }\frac{dV(x)}{%
V^{1+b}(x)} \\
&=&\frac{1}{bV^{b}(n_{0}+\frac{1}{2})}=\frac{1}{bV_{n_{0}}^{b}}.
\end{eqnarray*}%
Hence we have (\ref{2.9}).  The lemma is proved.
\end{proof}

\textbf{Note}. For example, $\nu _{n}=\frac{1}{(n-\tau)^{\beta }}%
(n\in \mathbf{N};0\leq \beta \leq 1,0\leq\tau<1)$ satisfies the conditions of Lemma 3
(for $n_{0}\geq 1)$.

\section{Equivalent Inequalities and Operator Expressions}

\begin{theorem}
 If $0\leq \alpha \leq \rho (\rho >0),0<\gamma
<\sigma \leq 1,k(\sigma )$ is given by (\ref{2.1}), then for $p>1,$ $%
0<||f||_{p,\Phi _{\delta }},||a||_{q,\widetilde{\Psi }}<\infty ,$ we have
the following equivalent inequalities:%
\begin{eqnarray}
I &:&=\sum_{n=1}^{\infty }\int_{0}^{\infty }\frac{\csc h(\rho (U^{\delta }(x)%
\widetilde{V}_{n})^{\gamma })}{e^{\alpha (U^{\delta }(x)\widetilde{V}%
_{n})^{\gamma }}}a_{n}f(x)dx<k(\sigma )||f||_{p,\Phi _{\delta }}||a||_{q,%
\widetilde{\Psi }},  \label{3.1} \\
J_{1} &:&=\sum_{n=1}^{\infty }\frac{\nu _{n}}{\widetilde{V}_{n}^{1-p\sigma }}%
\left[ \int_{0}^{\infty }\frac{\csc h(\rho (U^{\delta }(x)\widetilde{V}%
_{n})^{\gamma })}{e^{\alpha (U^{\delta }(x)\widetilde{V}_{n})^{\gamma }}}%
f(x)dx\right] ^{p}  \nonumber \\
&<&k(\sigma )||f||_{p,\Phi _{\delta }},  \label{3.2} \\
J_{2} &:&=\left\{ \int_{0}^{\infty }\frac{\mu (x)}{U^{1-q\delta \sigma }(x)}%
\left[ \sum_{n=1}^{\infty }\frac{\csc h(\rho (U^{\delta }(x)\widetilde{V}%
_{n})^{\gamma })}{e^{\alpha (U^{\delta }(x)\widetilde{V}_{n})^{\gamma }}}%
a_{n}\right] ^{q}dx\right\} ^{\frac{1}{q}}  \nonumber \\
&<&k(\sigma )||a||_{q,\widetilde{\Psi }}.  \label{3.3}
\end{eqnarray}
\end{theorem}

\begin{proof}
By H\"{o}lder's inequality with weight (cf. \cite{K1}), we have%
\begin{eqnarray}
&&\left[ \int_{0}^{\infty }\frac{\csc h(\rho (U^{\delta }(x)\widetilde{V}%
_{n})^{\gamma })}{e^{\alpha (U^{\delta }(x)\widetilde{V}_{n})^{\gamma }}}%
f(x)dx\right] ^{p}  \nonumber \\
&=&\left[ \int_{0}^{\infty }\frac{\csc h(\rho (U^{\delta }(x)\widetilde{V}%
_{n})^{\gamma })}{e^{\alpha (U^{\delta }(x)\widetilde{V}_{n})^{\gamma }}}%
\left( \frac{U^{\frac{1-\delta \sigma }{q}}(x)f(x)}{\widetilde{V}_{n}^{\frac{%
1-\sigma }{p}}\mu ^{\frac{1}{q}}(x)}\right) \left( \frac{\widetilde{V}_{n}^{%
\frac{1-\sigma }{p}}\mu ^{\frac{1}{q}}(x)}{U^{\frac{1-\delta \sigma }{q}}(x)}%
\right) dx\right] ^{p}  \nonumber \\
&\leq &\int_{0}^{\infty }\frac{\csc h(\rho (U^{\delta }(x)\widetilde{V}%
_{n})^{\gamma })}{e^{\alpha (U^{\delta }(x)\widetilde{V}_{n})^{\gamma }}}%
\left( \frac{U^{\frac{p(1-\delta \sigma )}{q}}(x)f^{p}(x)}{\widetilde{V}%
_{n}^{1-\sigma }\mu ^{\frac{p}{q}}(x)}\right) dx  \nonumber \\
&&\times \left[ \int_{0}^{\infty }\frac{\csc h(\rho (U^{\delta }(x)%
\widetilde{V}_{n})^{\gamma })}{e^{\alpha (U^{\delta }(x)\widetilde{V}%
_{n})^{\gamma }}}\frac{\widetilde{V}_{n}^{(1-\sigma )(p-1)}\mu (x)}{%
U^{1-\delta \sigma }(x)}dx\right] ^{p-1}  \nonumber \\
&=&\frac{(\varpi _{\delta }(\sigma ,n))^{p-1}}{\widetilde{V}_{n}^{p\sigma
-1}\nu _{n}}\int_{0}^{\infty }\frac{\csc h(\rho (U^{\delta }(x)\widetilde{V}%
_{n})^{\gamma })}{e^{\alpha (U^{\delta }(x)\widetilde{V}_{n})^{\gamma }}}%
\frac{U^{(1-\delta \sigma )(p-1)}(x)\nu _{n}}{\widetilde{V}_{n}^{1-\sigma
}\mu ^{p-1}(x)}f^{p}(x)dx.  \label{3.4}
\end{eqnarray}

In view of (\ref{2.6}) and the Lebesgue term by term integration theorem (cf.
\cite{K2}), we find%
\begin{eqnarray}
J_{1} &\leq &(k(\sigma ))^{\frac{1}{q}}\left[ \sum_{n=1}^{\infty
}\int_{0}^{\infty }\frac{\csc h(\rho (U^{\delta }(x)\widetilde{V}%
_{n})^{\gamma })}{e^{\alpha (U^{\delta }(x)\widetilde{V}_{n})^{\gamma }}}%
\frac{U^{(1-\delta \sigma )(p-1)}(x)\nu _{n}}{\widetilde{V}_{n}^{1-\sigma
}\mu ^{p-1}(x)}f^{p}(x)dx\right] ^{\frac{1}{p}}  \nonumber \\
&=&(k(\sigma ))^{\frac{1}{q}}\left[ \int_{0}^{\infty }\sum_{n=1}^{\infty }%
\frac{\csc h(\rho (U^{\delta }(x)\widetilde{V}_{n})^{\gamma })}{e^{\alpha
(U^{\delta }(x)\widetilde{V}_{n})^{\gamma }}}\frac{U^{(1-\delta \sigma
)(p-1)}(x)\nu _{n}}{\widetilde{V}_{n}^{1-\sigma }\mu ^{p-1}(x)}f^{p}(x)dx%
\right] ^{\frac{1}{p}}  \nonumber \\
&=&(k(\sigma ))^{\frac{1}{q}}\left[ \int_{0}^{\infty }\omega _{\delta
}(\sigma ,x)\frac{U^{p(1-\delta \sigma )-1}(x)}{\mu ^{p-1}(x)}f^{p}(x)dx%
\right] ^{\frac{1}{p}}.  \label{3.5}
\end{eqnarray}%
Then by (\ref{2.5}), we have (\ref{3.2}).

By H\"{o}lder's inequality (cf. \cite{K1}), we have%
\begin{eqnarray}
I &=&\sum_{n=1}^{\infty }\left[ \frac{\nu _{n}^{\frac{1}{p}}}{\widetilde{V}%
_{n}^{\frac{1}{p}-\sigma }}\int_{0}^{\infty }\frac{\csc h(\rho (U^{\delta
}(x)\widetilde{V}_{n})^{\gamma })}{e^{\alpha (U^{\delta }(x)\widetilde{V}%
_{n})^{\gamma }}}f(x)dx\right] \left( \frac{\widetilde{V}_{n}^{\frac{1}{p}%
-\sigma }a_{n}}{\nu _{n}^{\frac{1}{p}}}\right)  \nonumber \\
&\leq &J_{1}||a||_{q,\widetilde{\Psi }}.  \label{3.6}
\end{eqnarray}%
Then by (\ref{3.2}), we have (\ref{3.1}).

On the other hand, assuming that (\ref{3.1}) is valid, we set%
$$
a_{n}:=\frac{\nu _{n}}{\widetilde{V}_{n}^{1-p\sigma }}\left[
\int_{0}^{\infty }\frac{\csc h(\rho (U^{\delta }(x)\widetilde{V}%
_{n})^{\gamma })}{e^{\alpha (U^{\delta }(x)\widetilde{V}_{n})^{\gamma }}}%
f(x)dx\right] ^{p-1},n\in \mathbf{N}.
$$
Then, we find $J_{1}^{p}=||a||_{q,\widetilde{\Psi }}^{q}.$\\ 
If $J_{1}=0,$ then
(\ref{3.2}) is trivially valid.\\ 
If $J_{1}=\infty ,$ then (\ref{3.2}) keeps
impossible.\\ 
Suppose that $0<J_{1}<\infty .$ By (\ref{3.1}), it follows that%
\begin{eqnarray*}
||a||_{q,\widetilde{\Psi }}^{q} &=&J_{1}^{p}=I<k(\sigma )||f||_{p,\Phi
_{\delta }}||a||_{q,\widetilde{\Psi }}, \\
||a||_{q,\widetilde{\Psi }}^{q-1} &=&J_{1}<k(\sigma )||f||_{p,\Phi _{\delta
}},
\end{eqnarray*}%
and then (\ref{3.2}) follows, which is equivalent to (\ref{3.1}).

By H\"{o}lder's inequality with weight (cf. \cite{K1}), we obtain%
\begin{eqnarray}
&&\left[ \sum_{n=1}^{\infty }\frac{\csc h(\rho (U^{\delta }(x)\widetilde{V}%
_{n})^{\gamma })}{e^{\alpha (U^{\delta }(x)\widetilde{V}_{n})^{\gamma }}}%
a_{n}\right] ^{q}  \nonumber \\
&=&\left[ \sum_{n=1}^{\infty }\frac{\csc h(\rho (U^{\delta }(x)\widetilde{V}%
_{n})^{\gamma })}{e^{\alpha (U^{\delta }(x)\widetilde{V}_{n})^{\gamma }}}%
\left( \frac{U^{\frac{1-\delta \sigma }{q}}(x)\nu _{n}^{\frac{1}{p}}}{%
\widetilde{V}_{n}^{\frac{1-\sigma }{p}}}\right) \left( \frac{\widetilde{V}%
_{n}^{\frac{1-\sigma }{p}}a_{n}}{U^{\frac{1-\delta \sigma }{q}}(x)\nu _{n}^{%
\frac{1}{p}}}\right) \right] ^{q}  \nonumber \\
&\leq &\left[ \sum_{n=1}^{\infty }\frac{\csc h(\rho (U^{\delta }(x)%
\widetilde{V}_{n})^{\gamma })}{e^{\alpha (U^{\delta }(x)\widetilde{V}%
_{n})^{\gamma }}}\frac{U^{(1-\delta \sigma )(p-1)}(x)\nu _{n}}{\widetilde{V}%
_{n}^{1-\sigma }}\right] ^{q-1}  \nonumber \\
&&\times \sum_{n=1}^{\infty }\frac{\csc h(\rho (U^{\delta }(x)\widetilde{V}%
_{n})^{\gamma })}{e^{\alpha (U^{\delta }(x)\widetilde{V}_{n})^{\gamma }}}%
\frac{\widetilde{V}_{n}^{\frac{q(1-\sigma )}{p}}}{U^{1-\delta \sigma }(x)\nu
_{n}^{q-1}}a_{n}^{q}  \nonumber \\
&=&\frac{(\omega _{\delta }(\sigma ,x))^{q-1}}{U^{q\delta \sigma -1}(x)\mu
(x)}\sum_{n=1}^{\infty }\frac{\csc h(\rho (U^{\delta }(x)\widetilde{V}%
_{n})^{\gamma })}{e^{\alpha (U^{\delta }(x)\widetilde{V}_{n})^{\gamma }}}%
\frac{\widetilde{V}_{n}^{(1-\sigma )(q-1)}\mu (x)}{U^{1-\delta \sigma
}(x)\nu _{n}^{q-1}}a_{n}^{q}.  \label{3.7}
\end{eqnarray}%
Then by (\ref{2.5}) and Lebesgue term by term integration theorem (cf. \cite%
{K2}), it follows that%
\begin{eqnarray}
J_{2} &<&(k(\sigma ))^{\frac{1}{p}}\left\{ \int_{0}^{\infty
}\sum_{n=1}^{\infty }\frac{\csc h(\rho (U^{\delta }(x)\widetilde{V}%
_{n})^{\gamma })}{e^{\alpha (U^{\delta }(x)\widetilde{V}_{n})^{\gamma }}}%
\frac{\widetilde{V}_{n}^{(1-\sigma )(q-1)}\mu (x)}{U^{1-\delta \sigma
}(x)\nu _{n}^{q-1}}a_{n}^{q}dx\right\} ^{\frac{1}{q}}  \nonumber \\
&=&(k(\sigma ))^{\frac{1}{p}}\left\{ \sum_{n=1}^{\infty }\int_{0}^{\infty }%
\frac{\csc h(\rho (U^{\delta }(x)\widetilde{V}_{n})^{\gamma })}{e^{\alpha
(U^{\delta }(x)\widetilde{V}_{n})^{\gamma }}}\frac{\widetilde{V}%
_{n}^{(1-\sigma )(q-1)}\mu (x)}{U^{1-\delta \sigma }(x)\nu _{n}^{q-1}}%
a_{n}^{q}dx\right\} ^{\frac{1}{q}}  \nonumber \\
&=&(k(\sigma ))^{\frac{1}{p}}\left\{ \sum_{n=1}^{\infty }\varpi _{\delta
}(\sigma ,n)\frac{\widetilde{V}_{n}^{q(1-\sigma )-1}}{\nu _{n}^{q-1}}%
a_{n}^{q}\right\} ^{\frac{1}{q}}.  \label{3.8}
\end{eqnarray}%
Then by (\ref{2.6}), \ we have (\ref{3.3}).

By H\"{o}lder's inequality (cf. \cite{K1}), we have%
\begin{eqnarray}
I &=&\int_{0}^{\infty }\left( \frac{U^{\frac{1}{q}-\delta \sigma }(x)}{\mu ^{%
\frac{1}{q}}(x)}f(x)\right) \left[ \frac{\mu ^{\frac{1}{q}}(x)}{U^{\frac{1}{q%
}-\delta \sigma }(x)}\sum_{n=1}^{\infty }\frac{\csc h(\rho (U^{\delta }(x)%
\widetilde{V}_{n})^{\gamma })}{e^{\alpha (U^{\delta }(x)\widetilde{V}%
_{n})^{\gamma }}}a_{n}\right] dx  \nonumber \\
&\leq &||f||_{p,\Phi _{\delta }}J_{2}.  \label{3.9}
\end{eqnarray}%
Then by (\ref{3.3}), we have (\ref{3.1}).

On the other hand, assuming that (\ref{3.3}) is valid, we set%
$$
f(x):=\frac{\mu (x)}{U^{1-q\delta \sigma }(x)}\left[ \sum_{n=1}^{\infty }%
\frac{\csc h(\rho (U^{\delta }(x)\widetilde{V}_{n})^{\gamma })}{e^{\alpha
(U^{\delta }(x)\widetilde{V}_{n})^{\gamma }}}a_{n}\right] ^{q-1},\ x\in
\mathbf{R}_{+}.
$$
Then we find $J_{2}^{q}=||f||_{p,\Phi _{\delta }}^{p}.$\\ 
If $J_{2}=0,$ then (%
\ref{3.3}) is trivially valid.\\ 
If $J_{2}=\infty ,$ then (\ref{3.3}) keeps
impossible. \\
Suppose that $0<J_{2}<\infty .$ By (\ref{3.1}), it follows that%
\begin{eqnarray*}
||f||_{p,\Phi _{\delta }}^{p} &=&J_{2}^{q}=I<k(\sigma )||f||_{p,\Phi
_{\delta }}||a||_{q,\widetilde{\Psi }}, \\
||f||_{p,\Phi _{\delta }}^{p-1} &=&J_{2}<k(\sigma )||a||_{q,\widetilde{\Psi }%
},
\end{eqnarray*}%
and then (\ref{3.3}) follows, which is equivalent to (\ref{3.1}).

Therefore, (\ref{3.1}), (\ref{3.2}) and (\ref{3.3}) are equivalent.  The theorem is proved.
\end{proof}

\begin{theorem}
 With the assumptions of Theorem 1, if there exists $%
n_{0}\in \mathbf{N},$ such that $v_{n}\geq v_{n+1}$ $(n\in
\mathbf{\{}n_{0},n_{0}+1,\cdots \}),$ and $U(\infty )=V(\infty )=\infty ,$
then the constant factor $k(\sigma )$ in (\ref{3.1}), (\ref{3.2}) and (\ref%
{3.3}) is the best possible.
\end{theorem}

\begin{proof}
For $\varepsilon \in (0,q(\sigma -\gamma )),$ we set $
\widetilde{\sigma }=\sigma -\frac{\varepsilon }{q}(\in (\gamma ,1)),$ and $%
\widetilde{f}=\widetilde{f}(x),x\in \mathbf{R}_{+},\widetilde{a}=\{%
\widetilde{a}_{n}\}_{n=1}^{\infty },$%
\begin{eqnarray}
\widetilde{f}(x) &=&\left\{
\begin{array}{c}
U^{\delta (\widetilde{\sigma }+\varepsilon )-1}(x)\mu (x),0<x^{\delta }\leq 1
\\
0,x^{\delta }>0%
\end{array}%
\right. ,  \label{3.10} \\
\widetilde{a}_{n} &=&\widetilde{V}_{n}^{\widetilde{\sigma }-1}\nu _{n}=%
\widetilde{V}_{n}^{\sigma -\frac{\varepsilon }{q}-1}\nu _{n},n\in \mathbf{N}.
\label{3.11}
\end{eqnarray}%
Then for $\delta =\pm 1,$ since $U(\infty )=\infty ,$ we find%
\begin{equation}
\int_{\{x>0;0<x^{\delta }\leq 1\}}\frac{\mu (x)}{U^{1-\delta \varepsilon }(x)%
}dx=\frac{1}{\varepsilon }U^{\delta \varepsilon }(1).  \label{3.12}
\end{equation}%
By (\ref{2.9}), (\ref{3.12}) and (\ref{2.8}), we obtain%
\begin{eqnarray}
||\widetilde{f}||_{p,\Phi _{\delta }}||\widetilde{a}||_{q,\widetilde{\Psi }}
&=&\left( \int_{\{x>0;0<x^{\delta }\leq 1\}}\frac{\mu (x)dx}{U^{1-\delta
\varepsilon }(x)}\right) ^{\frac{1}{p}}\left( \sum_{n=1}^{\infty }\frac{\nu
_{n}}{\widetilde{V}_{n}^{1+\varepsilon }}\right) ^{\frac{1}{q}}  \nonumber \\
&=&\frac{1}{\varepsilon }U^{\frac{\delta \varepsilon }{p}}(1)\left( \frac{1}{%
V_{n_{0}}^{\varepsilon }}+\varepsilon \widetilde{O}(1)\right) ^{\frac{1}{q}},
\label{3.13}
\end{eqnarray}%
\begin{eqnarray*}
\widetilde{I} &:&=\int_{0}^{\infty }\sum_{n=1}^{\infty }\frac{\csc h(\rho
(U^{\delta }(x)\widetilde{V}_{n})^{\gamma })}{e^{\alpha (U^{\delta }(x)%
\widetilde{V}_{n})^{\gamma }}}\widetilde{a}_{n}\widetilde{f}(x)dx \\
&=&\int_{\{x>0;0<x^{\delta }\leq 1\}}\sum_{n=1}^{\infty }\frac{\csc h(\rho
(U^{\delta }(x)\widetilde{V}_{n})^{\gamma })}{e^{\alpha (U^{\delta }(x)%
\widetilde{V}_{n})^{\gamma }}}\frac{\widetilde{V}_{n}^{\widetilde{\sigma }%
-1}\nu _{n}\mu (x)}{U^{1-\delta (\widetilde{\sigma }+\varepsilon )}(x)}dx \\
&=&\int_{\{x>0;0<x^{\delta }\leq 1\}}\omega _{\delta }(\widetilde{\sigma },x)%
\frac{\mu (x)}{U^{1-\delta \varepsilon }(x)}dx \\
&\geq &k(\widetilde{\sigma })\int_{\{x>0;0<x^{\delta }\leq 1\}}(1-\theta
_{\delta }(\widetilde{\sigma },x))\frac{\mu (x)}{U^{1-\delta \varepsilon }(x)%
}dx \\
&=&k(\widetilde{\sigma })\int_{\{x>0;0<x^{\delta }\leq
1\}}(1-O((U(x))^{\delta (\sigma -\frac{\varepsilon }{q}-\gamma )}))\frac{\mu
(x)}{U^{1-\delta \varepsilon }(x)}dx \\
&=&k(\widetilde{\sigma })\left[ \int_{\{x>0;0<x^{\delta }\leq 1\}}\frac{\mu
(x)}{U^{1-\delta \varepsilon }(x)}dx\right. \\
&&\left. -\int_{\{x>0;0<x^{\delta }\leq 1\}}O(\frac{\mu (x)}{U^{1-\delta
(\sigma -\gamma +\frac{\varepsilon }{p})}(x)})dx\right] \\
&=&\frac{1}{\varepsilon }k(\sigma -\frac{\varepsilon }{q})(U^{\delta
\varepsilon }(1)-\varepsilon O(1)).
\end{eqnarray*}

If there exists a positive constant $K\leq k(\sigma ),$ such that (\ref{3.1}%
) is valid when replacing $k(\sigma )$ to $K,$ then in particular, by
Lebesgue term by term integration theorem, we have $\varepsilon \widetilde{I}%
<\varepsilon K||\widetilde{f}||_{p,\Phi _{\delta }}||\widetilde{a}||_{q,%
\widetilde{\Psi }},$ namely,%
$$
k(\sigma -\frac{\varepsilon }{q})(U^{\delta \varepsilon }(1)-\varepsilon
O(1))<K\cdot U^{\frac{\delta \varepsilon }{p}}(1)\left( \frac{1}{%
V_{n_{0}}^{\varepsilon }}+\varepsilon \widetilde{O}(1)\right) ^{\frac{1}{q}}.
$$
It follows that $k(\sigma )\leq K(\varepsilon \rightarrow 0^{+}).$ Hence, $%
K=k(\sigma )$ is the best possible constant factor of (\ref{3.1}).

The constant factor $k(\sigma )$ in (\ref{3.2}) ((\ref{3.3})) is still the
best possible. Otherwise, we would reach a contradiction by (\ref{3.6}) ((%
\ref{3.9})) that the constant factor in (\ref{3.1}) is not the best
possible. The theorem is proved.
\end{proof}

For $p>1,$ we find%
$$
\widetilde{\Psi }^{1-p}(n)=\frac{\nu _{n}}{\widetilde{V}_{n}^{1-p\sigma }}%
(n\in \mathbf{N}),\Phi _{\delta }^{1-q}(x)=\frac{\mu (x)}{U^{1-q\delta
\sigma }(x)}(x\in \mathbf{R}_{+}),
$$
and define the following real normed spaces:%
\begin{eqnarray*}
L_{p,\Phi _{\delta }}(\mathbf{R}_{+}) &=&\{f;f=f(x),x\in \mathbf{R}%
_{+},||f||_{p,\Phi _{\delta }}<\infty \}, \\
l_{q,\widetilde{\Psi }} &=&\{a;a=\{a_{n}\}_{n=1}^{\infty },||a||_{q,%
\widetilde{\Psi }}<\infty \}, \\
L_{q,\Phi _{\delta }^{1-q}}(\mathbf{R}_{+}) &=&\{h;h=h(x),x\in \mathbf{R}%
_{+},||h||_{q,\Phi _{\delta }^{1-q}}<\infty \}, \\
l_{p,\widetilde{\Psi }^{1-p}} &=&\{c;c=\{c_{n}\}_{n=1}^{\infty },||c||_{p,%
\widetilde{\Psi }^{1-p}}<\infty \}.
\end{eqnarray*}

Assuming that $f\in L_{p,\Phi _{\delta }}(\mathbf{R}_{+}),$ setting
$$
c=\{c_{n}\}_{n=1}^{\infty },c_{n}:=\int_{0}^{\infty }\frac{\csc h(\rho
(U^{\delta }(x)\widetilde{V}_{n})^{\gamma })}{e^{\alpha (U^{\delta }(x)%
\widetilde{V}_{n})^{\gamma }}}f(x)dx,n\in \mathbf{N},
$$
we can rewrite (\ref{3.2}) as follows:
$$
||c||_{p,\widetilde{\Psi }^{1-p}}<k(\sigma )||f||_{p,\Phi _{\delta }}<\infty
,
$$
namely, $c\in l_{p,\widetilde{\Psi }^{1-p}}.$

\begin{definition}
 Define a half-discrete Hardy-Hilbert-type operator $$
T_{1}:L_{p,\Phi _{\delta }}(\mathbf{R}_{+})\rightarrow l_{p,\widetilde{\Psi }%
^{1-p}}$$ as follows:\\ 
For any $f\in L_{p,\Phi _{\delta }}(\mathbf{R}_{+}),$
there exists a unique representation $T_{1}f=c\in l_{p,\widetilde{\Psi }%
^{1-p}}.$ Define the formal inner product of $T_{1}f$ and $%
a=\{a_{n}\}_{n=1}^{\infty }\in l_{q,\widetilde{\Psi }}$ as follows:%
\begin{equation}
(T_{1}f,a):=\sum_{n=1}^{\infty }\left[ \int_{0}^{\infty }\frac{\csc h(\rho
(U^{\delta }(x)\widetilde{V}_{n})^{\gamma })}{e^{\alpha (U^{\delta }(x)%
\widetilde{V}_{n})^{\gamma }}}f(x)dx\right] a_{n}.  \label{3.14}
\end{equation}
\end{definition}

Then we can rewrite (\ref{3.1}) and (\ref{3.2}) as:%
\begin{eqnarray}
(T_{1}f,a) &<&k(\sigma )||f||_{p,\Phi _{\delta }}||a||_{q,\widetilde{\Psi }},
\label{3.15} \\
||T_{1}f||_{p,\widetilde{\Psi }^{1-p}} &<&k(\sigma )||f||_{p,\Phi _{\delta
}}.  \label{3.16}
\end{eqnarray}

Define the norm of operator $T_{1}$ as follows:%
$$
||T_{1}||:=\sup_{f(\neq \theta )\in L_{p,\Phi _{\delta }}(\mathbf{R}_{+})}%
\frac{||T_{1}f||_{p,\widetilde{\Psi }^{1-p}}}{||f||_{p,\Phi _{\delta }}}.
$$
Then by (\ref{3.16}), it is evident that $||T_{1}||\leq k(\sigma ).$ Since by
Theorem 2, the constant factor in (\ref{3.16}) is the best possible, we have%
\begin{equation}
||T_{1}||=k(\sigma )=\frac{2\Gamma (\frac{\sigma }{\gamma })}{\gamma (2\rho
)^{\sigma /\gamma }}\zeta (\frac{\sigma }{\gamma },\frac{\alpha +\rho }{%
2\rho }).  \label{3.17}
\end{equation}

Assuming that $a=\{a_{n}\}_{n=1}^{\infty }\in l_{q,\widetilde{\Psi }},$
setting
$$
h(x):=\sum_{n=1}^{\infty }\frac{\csc h(\rho (U^{\delta }(x)\widetilde{V}%
_{n})^{\gamma })}{e^{\alpha (U^{\delta }(x)\widetilde{V}_{n})^{\gamma }}}%
a_{n},x\in \mathbf{R}_{+},
$$
we can rewrite (\ref{3.3}) as follows:
$$
||h||_{q,\Phi _{\delta }^{1-q}}<k(\sigma )||a||_{q,\widetilde{\Psi }}<\infty
,
$$
namely, $h\in L_{q,\Phi _{\delta }^{1-q}}(\mathbf{R}_{+}).$

\begin{definition}
 Define a half-discrete Hardy-Hilbert-type operator $$
T_{2}:l_{q,\widetilde{\Psi }}\rightarrow L_{q,\Phi _{\delta }^{1-q}}(\mathbf{%
R}_{+})$$ as follows:\\ 
For any $a=\{a_{n}\}_{n=1}^{\infty }\in l_{q,\widetilde{%
\Psi }},$ there exists a unique representation $T_{2}a=h\in L_{q,\Phi
_{\delta }^{1-q}}(\mathbf{R}_{+}).$ Define the formal inner product of $%
T_{2}a$ and $f\in L_{p,\Phi _{\delta }}(\mathbf{R}_{+})$ by:%
\begin{equation}
(T_{2}a,f):=\int_{0}^{\infty }\left[ \sum_{n=1}^{\infty }\frac{\csc h(\rho
(U^{\delta }(x)\widetilde{V}_{n})^{\gamma })}{e^{\alpha (U^{\delta }(x)%
\widetilde{V}_{n})^{\gamma }}}a_{n}\right] f(x)dx.  \label{3.18}
\end{equation}
\end{definition}

Then we can rewrite (\ref{3.1}) and (\ref{3.3}) as follows:%
\begin{eqnarray}
(T_{2}a,f) &<&k(\sigma )||f||_{p,\Phi _{\delta }}||a||_{q,\widetilde{\Psi }},
\label{3.19} \\
||T_{2}a||_{q,\Phi _{\delta }^{1-q}} &<&k(\sigma )||a||_{q,\widetilde{\Psi }%
}.  \label{3.20}
\end{eqnarray}

Define the norm of operator $T_{2}$ by:%
$$
||T_{2}||:=\sup_{a(\neq \theta )\in l_{q,\widetilde{\Psi }}}\frac{%
||T_{2}a||_{q,\Phi _{\delta }^{1-q}}}{||a||_{q,\widetilde{\Psi }}}.
$$
Then by (\ref{3.20}), we find $||T_{2}||\leq k(\sigma ).$ Since by Theorem
2, the constant factor in (\ref{3.20}) is the best possible, we have%
\begin{equation}
||T_{2}||=k(\sigma )=\frac{2\Gamma (\frac{\sigma }{\gamma })}{\gamma (2\rho
)^{\sigma /\gamma }}\zeta (\frac{\sigma }{\gamma },\frac{\alpha +\rho }{%
2\rho })=||T_{1}||.  \label{3.21}
\end{equation}

\section{Some Equivalent Reverses}

In the following, we also set
$$
\widetilde{\Phi }_{\delta }(x):=(1-\theta _{\delta }(\sigma ,x))\frac{%
U^{p(1-\delta \sigma )-1}(x)}{\mu ^{p-1}(x)}(x\in \mathbf{R}_{+}).$$
For $0<p<1$ or $p<0,$ we still use the formal symbols $%
||f||_{p,\Phi _{\delta }}$, $||f||_{p,\widetilde{\Phi }_{\delta }}$ and $%
||a||_{q,\widetilde{\Psi }}.$

\begin{theorem}
 If $0\leq \alpha \leq \rho (\rho >0),0<\gamma <\sigma
\leq 1,k(\sigma )$ is given by (\ref{2.1}), there exists $n_{0}\in
\mathbf{N},$ such that $v_{n}\geq v_{n+1}$ $(n\in \mathbf{\{}%
n_{0},n_{0}+1,\cdots \}),$ and $U(\infty )=V(\infty )=\infty ,$ then for $%
p<0,$ $0<||f||_{p,\Phi _{\delta }},||a||_{q,\widetilde{\Psi }}<\infty ,$ we
have the following equivalent inequalities with the best possible constant
factor $k(\sigma )$:%
\begin{eqnarray}
I &=&\sum_{n=1}^{\infty }\int_{0}^{\infty }\frac{\csc h(\rho (U^{\delta }(x)%
\widetilde{V}_{n})^{\gamma })}{e^{\alpha (U^{\delta }(x)\widetilde{V}%
_{n})^{\gamma }}}a_{n}f(x)dx>k(\sigma )||f||_{p,\Phi _{\delta }}||a||_{q,%
\widetilde{\Psi }},  \label{4.1} \\
J_{1} &=&\sum_{n=1}^{\infty }\frac{\nu _{n}}{\widetilde{V}_{n}^{1-p\sigma }}%
\left[ \int_{0}^{\infty }\frac{\csc h(\rho (U^{\delta }(x)\widetilde{V}%
_{n})^{\gamma })}{e^{\alpha (U^{\delta }(x)\widetilde{V}_{n})^{\gamma }}}%
f(x)dx\right] ^{p}>k(\sigma )||f||_{p,\Phi _{\delta }},  \label{4.2} \\
J_{2} &=&\left\{ \int_{0}^{\infty }\frac{\mu (x)}{U^{1-q\delta \sigma }(x)}%
\left[ \sum_{n=1}^{\infty }\frac{\csc h(\rho (U^{\delta }(x)\widetilde{V}%
_{n})^{\gamma })}{e^{\alpha (U^{\delta }(x)\widetilde{V}_{n})^{\gamma }}}%
a_{n}\right] ^{q}dx\right\} ^{\frac{1}{q}}  \nonumber \\
&>&k(\sigma )||a||_{q,\widetilde{\Psi }}.  \label{4.3}
\end{eqnarray}
\end{theorem}

\begin{proof}
By the reverse H\"{o}lder's inequality with weight (cf. \cite%
{K1}), since $p<0,$ similarly to the way we obtained (\ref{3.4}) and (\ref%
{3.5}), we have%
\begin{eqnarray*}
&&\left[ \int_{0}^{\infty }\frac{\csc h(\rho (U^{\delta }(x)\widetilde{V}%
_{n})^{\gamma })}{e^{\alpha (U^{\delta }(x)\widetilde{V}_{n})^{\gamma }}}%
f(x)dx\right] ^{p} \\
&\leq &\frac{\widetilde{V}_{n}^{1-p\sigma }}{(\varpi _{\delta }(\sigma
,n))^{1-p}\nu _{n}}\int_{0}^{\infty }\frac{\csc h(\rho (U^{\delta }(x)%
\widetilde{V}_{n})^{\gamma })}{e^{\alpha (U^{\delta }(x)\widetilde{V}%
_{n})^{\gamma }}}\frac{U^{(1-\delta \sigma )(p-1)}(x)\nu _{n}}{\widetilde{V}%
_{n}^{1-\sigma }\mu ^{p-1}(x)}f^{p}(x)dx,
\end{eqnarray*}%
and then by (\ref{2.7}) and Lebesgue term by term integration theorem, it
follows that%
\begin{eqnarray*}
J_{1} &\geq &(k(\sigma ))^{\frac{1}{q}}\left[ \sum_{n=1}^{\infty
}\int_{0}^{\infty }\frac{\csc h(\rho (U^{\delta }(x)\widetilde{V}%
_{n})^{\gamma })}{e^{\alpha (U^{\delta }(x)\widetilde{V}_{n})^{\gamma }}}%
\frac{U^{(1-\delta \sigma )(p-1)}(x)\nu _{n}}{\widetilde{V}_{n}^{1-\sigma
}\mu ^{p-1}(x)}f^{p}(x)dx\right] ^{\frac{1}{p}} \\
&=&(k(\sigma ))^{\frac{1}{q}}\left[ \int_{0}^{\infty }\omega _{\delta
}(\sigma ,x)\frac{U^{p(1-\delta \sigma )-1}(x)}{\mu ^{p-1}(x)}f^{p}(x)dx%
\right] ^{\frac{1}{p}}.
\end{eqnarray*}%
Then by (\ref{2.5}), we have (\ref{4.2}).

By the reverse H\"{o}lder's inequality (cf. \cite{K1}), we have%
\begin{eqnarray}
I &=&\sum_{n=1}^{\infty }\left[ \frac{\nu _{n}^{\frac{1}{p}}}{\widetilde{V}%
_{n}^{\frac{1}{p}-\sigma }}\int_{0}^{\infty }\frac{\csc h(\rho (U^{\delta
}(x)\widetilde{V}_{n})^{\gamma })}{e^{\alpha (U^{\delta }(x)\widetilde{V}%
_{n})^{\gamma }}}f(x)dx\right] \left( \frac{\widetilde{V}_{n}^{\frac{1}{p}%
-\sigma }a_{n}}{\nu _{n}^{\frac{1}{p}}}\right)  \nonumber \\
&\geq &J_{1}||a||_{q,\widetilde{\Psi }}.  \label{4.4}
\end{eqnarray}%
Then by (\ref{4.2}), we have (\ref{4.1}).

On the other hand, assuming that (\ref{4.1}) is valid, we set $a_{n}$ as in
Theorem 1. Then we find $J_{1}^{p}=||a||_{q,\widetilde{\Psi }}^{q}.$ \\
If $%
J_{1}=\infty ,$ then (\ref{4.2}) is trivially valid. \\
If $J_{1}=0,$ then (\ref%
{4.2}) is impossible.\\ 
Suppose that $0<J_{1}<\infty .$ By (\ref{4.1}), it
follows that%
\begin{eqnarray*}
||a||_{q,\widetilde{\Psi }}^{q} &=&J_{1}^{p}=I>k(\sigma )||f||_{p,\Phi
_{\delta }}||a||_{q,\widetilde{\Psi }}, \\
||a||_{q,\widetilde{\Psi }}^{q-1} &=&J_{1}>k(\sigma )||f||_{p,\Phi _{\delta
}},
\end{eqnarray*}%
and then (\ref{4.2}) follows, which is equivalent to (\ref{4.1}).

By the reverse of H\"{o}lder's inequality with weight (cf. \cite{K1}),
since $0<q<1,$ similarly to the way we obtained (\ref{3.7}) and (\ref{3.8}),
we have%
\begin{eqnarray*}
&&\left[ \sum_{n=1}^{\infty }\frac{\csc h(\rho (U^{\delta }(x)\widetilde{V}%
_{n})^{\gamma })}{e^{\alpha (U^{\delta }(x)\widetilde{V}_{n})^{\gamma }}}%
a_{n}\right] ^{q} \\
&\geq &\frac{(\omega _{\delta }(\sigma ,x))^{q-1}}{U^{q\delta \sigma
-1}(x)\mu (x)}\sum_{n=1}^{\infty }\frac{\csc h(\rho (U^{\delta }(x)%
\widetilde{V}_{n})^{\gamma })}{e^{\alpha (U^{\delta }(x)\widetilde{V}%
_{n})^{\gamma }}}\frac{\widetilde{V}_{n}^{(1-\sigma )(q-1)}\mu (x)}{%
U^{1-\delta \sigma }(x)\nu _{n}^{q-1}}a_{n}^{q},
\end{eqnarray*}%
and then by (\ref{2.5}) and Lebesgue term by term integration theorem, it
follows that%
\begin{eqnarray*}
J_{2} &>&(k(\sigma ))^{\frac{1}{p}}\left[ \int_{0}^{\infty
}\sum_{n=1}^{\infty }\frac{\csc h(\rho (U^{\delta }(x)\widetilde{V}%
_{n})^{\gamma })}{e^{\alpha (U^{\delta }(x)\widetilde{V}_{n})^{\gamma }}}%
\frac{\widetilde{V}_{n}^{(1-\sigma )(q-1)}\mu (x)}{U^{1-\delta \sigma
}(x)\nu _{n}^{q-1}}a_{n}^{q}dx\right] ^{\frac{1}{q}} \\
&=&(k(\sigma ))^{\frac{1}{p}}\left[ \sum_{n=1}^{\infty }\varpi _{\delta
}(\sigma ,n)\frac{\widetilde{V}_{n}^{q(1-\sigma )-1}}{\nu _{n}^{q-1}}%
a_{n}^{q}\right] ^{\frac{1}{q}}.
\end{eqnarray*}%
Then by (\ref{2.7}), \ we obtain (\ref{4.3}).

By the reverse H\"{o}lder's inequality (cf. \cite{K1}), we get%
\begin{eqnarray}
I &=&\int_{0}^{\infty }\left( \frac{U^{\frac{1}{q}-\delta \sigma }(x)}{\mu ^{%
\frac{1}{q}}(x)}f(x)\right) \left[ \frac{\mu ^{\frac{1}{q}}(x)}{U^{\frac{1}{q%
}-\delta \sigma }(x)}\sum_{n=1}^{\infty }\frac{\csc h(\rho (U^{\delta }(x)%
\widetilde{V}_{n})^{\gamma })}{e^{\alpha (U^{\delta }(x)\widetilde{V}%
_{n})^{\gamma }}}a_{n}\right] dx   \nonumber \\
&\geq &||f||_{p,\Phi _{\delta }}J_{2}.  \label{4.5}
\end{eqnarray}%
Then by (\ref{4.3}), we derive (\ref{4.1}).

On the other hand, assuming that (\ref{4.3}) is valid, we set $f(x)$ as in
Theorem 1. Then we find $J_{2}^{q}=||f||_{p,\Phi _{\delta }}^{p}.$\\ 
If $%
J_{2}=\infty ,$ then (\ref{4.3}) is trivially valid. \\
If $J_{2}=0,$ then (\ref%
{4.3}) keeps impossible.\\ 
Suppose that $0<J_{2}<\infty .$ By (\ref{4.1}), it
follows that%
\begin{eqnarray*}
||f||_{p,\Phi _{\delta }}^{p} &=&J_{2}^{q}=I>k(\sigma )||f||_{p,\Phi
_{\delta }}||a||_{q,\widetilde{\Psi }}, \\
||f||_{p,\Phi _{\delta }}^{p-1} &=&J_{2}>k(\sigma )||a||_{q,\widetilde{\Psi }%
},
\end{eqnarray*}%
and then (\ref{4.3}) follows, which is equivalent to (\ref{4.1}).

Therefore, inequalities (\ref{4.1}), (\ref{4.2}) and (\ref{4.3}) are
equivalent.

For $\varepsilon \in (0,q(\sigma -\gamma )),$ we set $\widetilde{\sigma }%
=\sigma -\frac{\varepsilon }{q}(\in (\gamma ,1)),$ and $\widetilde{f}=%
\widetilde{f}(x),x\in \mathbf{R}_{+},\widetilde{a}=\{\widetilde{a}%
_{n}\}_{n=1}^{\infty },$%
\begin{eqnarray*}
\widetilde{f}(x) &=&\left\{
\begin{array}{c}
U^{\delta (\widetilde{\sigma }+\varepsilon )-1}(x)\mu (x),0<x^{\delta }\leq 1
\\
0,x^{\delta }>0%
\end{array}%
\right. , \\
\widetilde{a}_{n} &=&\widetilde{V}_{n}^{\widetilde{\sigma }-1}\nu _{n}=%
\widetilde{V}_{n}^{\sigma -\frac{\varepsilon }{q}-1}\nu _{n},n\in \mathbf{N}.
\end{eqnarray*}%
By (\ref{2.9}), (\ref{3.12}) and (\ref{2.5}), we obtain%
\begin{equation}
||\widetilde{f}||_{p,\Phi _{\delta }}||\widetilde{a}||_{q,\widetilde{\Psi }}=%
\frac{1}{\varepsilon }U^{\frac{\delta \varepsilon }{p}}(1)\left( \frac{1}{%
V_{n_{0}}^{\varepsilon }}+\varepsilon \widetilde{O}(1)\right) ^{\frac{1}{q}},
\end{equation}%
\begin{eqnarray*}
\widetilde{I} &=&\sum_{n=1}^{\infty }\int_{0}^{\infty }\frac{\csc h(\rho
(U^{\delta }(x)\widetilde{V}_{n})^{\gamma })}{e^{\alpha (U^{\delta }(x)%
\widetilde{V}_{n})^{\gamma }}}\widetilde{a}_{n}\widetilde{f}(x)dx \\
&=&\int_{\{x>0;0<x^{\delta }\leq 1\}}\omega _{\delta }(\widetilde{\sigma },x)%
\frac{\mu (x)}{U^{1-\delta \varepsilon }(x)}dx \\
&\leq &k(\widetilde{\sigma })\int_{\{x>0;0<x^{\delta }\leq 1\}}\frac{\mu (x)%
}{U^{1-\delta \varepsilon }(x)}dx=\frac{1}{\varepsilon }k(\sigma -\frac{%
\varepsilon }{q})U^{\delta \varepsilon }(1).
\end{eqnarray*}

If there exists a positive constant $K\geq k(\sigma ),$ such that (\ref{4.1}%
) is valid when replacing $k(\sigma )$ by $K,$ then in particular, we have $%
\varepsilon \widetilde{I}>\varepsilon K||\widetilde{f}||_{p,\Phi _{\delta
}}||\widetilde{a}||_{q,\widetilde{\Psi }},$ namely,%
$$
k(\sigma -\frac{\varepsilon }{q})U^{\delta \varepsilon }(1)>K\cdot U^{\frac{%
\delta \varepsilon }{p}}(1)\left( \frac{1}{V_{n_{0}}^{\varepsilon }}%
+\varepsilon \widetilde{O}(1)\right) ^{\frac{1}{q}}.
$$
It follows that $k(\sigma )\geq K(\varepsilon \rightarrow 0^{+}).$ Hence, $%
K=k(\sigma )$ is the best possible constant factor of (\ref{4.1}).

The constant factor $k(\sigma )$ in (\ref{4.2}) ((\ref{4.3})) is still the
best possible. Otherwise, we would reach a contradiction by (\ref{4.4}) ((%
\ref{4.5})) that the constant factor in (\ref{4.1}) is not the best
possible. The theorem is proved.
\end{proof}

\begin{theorem}
 With the assumptions of Theorem 3, if $$0<p<1,\
0<||f||_{p,\Phi _{\delta }},\ ||a||_{q,\widetilde{\Psi }}<\infty ,$$ then we
have the following equivalent inequalities with the best possible constant
factor $k(\sigma )$:%
\begin{eqnarray}
I &=&\sum_{n=1}^{\infty }\int_{0}^{\infty }\frac{\csc h(\rho (U^{\delta }(x)%
\widetilde{V}_{n})^{\gamma })}{e^{\alpha (U^{\delta }(x)\widetilde{V}%
_{n})^{\gamma }}}a_{n}f(x)dx>k(\sigma )||f||_{p,\widetilde{\Phi }_{\delta
}}||a||_{q,\widetilde{\Psi }},  \label{4.6} \\
J_{1} &=&\sum_{n=1}^{\infty }\frac{\nu _{n}}{\widetilde{V}_{n}^{1-p\sigma }}%
\left[ \int_{0}^{\infty }\frac{\csc h(\rho (U^{\delta }(x)\widetilde{V}%
_{n})^{\gamma })}{e^{\alpha (U^{\delta }(x)\widetilde{V}_{n})^{\gamma }}}%
f(x)dx\right] ^{p}>k(\sigma )||f||_{p,\widetilde{\Phi }_{\delta }},
\label{4.7} \\
J &:&=\left\{ \int_{0}^{\infty }\frac{(1-\theta _{\delta }(\sigma
,x))^{1-q}\mu (x)}{U^{1-q\delta \sigma }(x)}\left[ \sum_{n=1}^{\infty }\frac{%
\csc h(\rho (U^{\delta }(x)\widetilde{V}_{n})^{\gamma })}{e^{\alpha
(U^{\delta }(x)\widetilde{V}_{n})^{\gamma }}}a_{n}\right] ^{q}dx\right\} ^{%
\frac{1}{q}}   \nonumber \\
&>&k(\sigma )||a||_{q,\widetilde{\Psi }}.  \label{4.8}
\end{eqnarray}
\end{theorem}

\begin{proof}
By the reverse H\"{o}lder's inequality with weight (cf. \cite%
{K1}), since $0<p<1,$ similarly to the way we obtained (\ref{3.4}) and (\ref%
{3.5}), we have%
\begin{eqnarray*}
&&\left[ \int_{0}^{\infty }\frac{\csc h(\rho (U^{\delta }(x)\widetilde{V}%
_{n})^{\gamma })}{e^{\alpha (U^{\delta }(x)\widetilde{V}_{n})^{\gamma }}}%
f(x)dx\right] ^{p} \\
&\geq &\frac{(\varpi _{\delta }(\sigma ,n))^{p-1}}{\widetilde{V}%
_{n}^{p\sigma -1}\nu _{n}}\int_{0}^{\infty }\frac{\csc h(\rho (U^{\delta }(x)%
\widetilde{V}_{n})^{\gamma })}{e^{\alpha (U^{\delta }(x)\widetilde{V}%
_{n})^{\gamma }}}\frac{U^{(1-\delta \sigma )(p-1)}(x)\nu _{n}}{\widetilde{V}%
_{n}^{1-\sigma }\mu ^{p-1}(x)}f^{p}(x)dx,
\end{eqnarray*}%
and then in view of (\ref{2.7}) and Lebesgue term by term integration theorem, we find%
\begin{eqnarray*}
J_{1} &\geq &(k(\sigma ))^{\frac{1}{q}}\left[ \sum_{n=1}^{\infty
}\int_{0}^{\infty }\frac{\csc h(\rho (U^{\delta }(x)\widetilde{V}%
_{n})^{\gamma })}{e^{\alpha (U^{\delta }(x)\widetilde{V}_{n})^{\gamma }}}%
\frac{U^{(1-\delta \sigma )(p-1)}(x)\nu _{n}}{\widetilde{V}_{n}^{1-\sigma
}\mu ^{p-1}(x)}f^{p}(x)dx\right] ^{\frac{1}{p}} \\
&=&(k(\sigma ))^{\frac{1}{q}}\left[ \int_{0}^{\infty }\omega _{\delta
}(\sigma ,x)\frac{U^{p(1-\delta \sigma )-1}(x)}{\mu ^{p-1}(x)}f^{p}(x)dx%
\right] ^{\frac{1}{p}}.
\end{eqnarray*}%
Then by (\ref{2.8}), we have (\ref{4.7}).

By the reverse H\"{o}lder's inequality (cf. \cite{K1}), we have%
\begin{eqnarray}
I &=&\sum_{n=1}^{\infty }\left[ \frac{\nu _{n}^{\frac{1}{p}}}{\widetilde{V}%
_{n}^{\frac{1}{p}-\sigma }}\int_{0}^{\infty }\frac{\csc h(\rho (U^{\delta
}(x)\widetilde{V}_{n})^{\gamma })}{e^{\alpha (U^{\delta }(x)\widetilde{V}%
_{n})^{\gamma }}}f(x)dx\right] \left( \frac{\widetilde{V}_{n}^{\frac{1}{p}%
-\sigma }a_{n}}{\nu _{n}^{\frac{1}{p}}}\right)  \nonumber \\
&\geq &J_{1}||a||_{q,\widetilde{\Psi }}.  \label{4.9}
\end{eqnarray}%
Then by (\ref{4.7}), we have (\ref{4.6}).

On the other hand, Assuming that (\ref{4.6}) is valid, we set $a_{n}$ as in
Theorem 1. Then we find $J_{1}^{p}=||a||_{q,\widetilde{\Psi }}^{q}.$ \\
If $%
J_{1}=\infty ,$ then (\ref{4.7}) is trivially valid. \\
If $J_{1}=0,$ then (\ref%
{4.7}) remains impossible. \\
Suppose that $0<J_{1}<\infty .$ By (\ref{4.6}), it
follows that%
\begin{eqnarray*}
||a||_{q,\widetilde{\Psi }}^{q} &=&J_{1}^{p}=I>k(\sigma )||f||_{p,\widetilde{%
\Phi }_{\delta }}||a||_{q,\widetilde{\Psi }}, \\
||a||_{q,\widetilde{\Psi }}^{q-1} &=&J_{1}>k(\sigma )||f||_{p,\widetilde{%
\Phi }_{\delta }},
\end{eqnarray*}%
and then (\ref{4.7}) follows, which is equivalent to (\ref{4.6}).

By the reverse H\"{o}lder's inequality with weight (cf. \cite{K1}),
since $q<0,$ we have%
\begin{eqnarray*}
&&\left[ \sum_{n=1}^{\infty }\frac{\csc h(\rho (U^{\delta }(x)\widetilde{V}%
_{n})^{\gamma })}{e^{\alpha (U^{\delta }(x)\widetilde{V}_{n})^{\gamma }}}%
a_{n}\right] ^{q} \\
&\leq &\frac{(\omega _{\delta }(\sigma ,x))^{q-1}}{U^{q\delta \sigma
-1}(x)\mu (x)}\sum_{n=1}^{\infty }\frac{\csc h(\rho (U^{\delta }(x)%
\widetilde{V}_{n})^{\gamma })}{e^{\alpha (U^{\delta }(x)\widetilde{V}%
_{n})^{\gamma }}}\frac{\widetilde{V}_{n}^{(1-\sigma )(q-1)}\mu (x)}{%
U^{1-\delta \sigma }(x)\nu _{n}^{q-1}}a_{n}^{q},
\end{eqnarray*}%
and then by (\ref{2.8}) and Lebesgue term by term integration theorem, it
follows that%
\begin{eqnarray*}
J &>&(k(\sigma ))^{\frac{1}{p}}\left[ \int_{0}^{\infty }\sum_{n=1}^{\infty }%
\frac{\csc h(\rho (U^{\delta }(x)\widetilde{V}_{n})^{\gamma })}{e^{\alpha
(U^{\delta }(x)\widetilde{V}_{n})^{\gamma }}}\frac{\widetilde{V}%
_{n}^{(1-\sigma )(q-1)}\mu (x)}{U^{1-\delta \sigma }(x)\nu _{n}^{q-1}}%
a_{n}^{q}dx\right] ^{\frac{1}{q}} \\
&=&(k(\sigma ))^{\frac{1}{p}}\left[ \sum_{n=1}^{\infty }\varpi _{\delta
}(\sigma ,n)\frac{\widetilde{V}_{n}^{q(1-\sigma )-1}}{\nu _{n}^{q-1}}%
a_{n}^{q}\right] ^{\frac{1}{q}}.
\end{eqnarray*}%
Then by (\ref{2.7}), \ we have (\ref{4.8}).

By the reverse H\"{o}lder's inequality (cf. \cite{K1}), we have%
\begin{eqnarray}
I &=&\int_{0}^{\infty }\left[ (1-\theta _{\delta }(\sigma ,x))^{\frac{1}{p}}%
\frac{U^{\frac{1}{q}-\delta \sigma }(x)}{\mu ^{\frac{1}{q}}(x)}f(x)\right]
\nonumber \\
&&\times \left[ \frac{(1-\theta _{\delta }(\sigma ,x))^{\frac{-1}{p}}\mu ^{%
\frac{1}{q}}(x)}{U^{\frac{1}{q}-\delta \sigma }(x)}\sum_{n=1}^{\infty }\frac{%
\csc h(\rho (U^{\delta }(x)\widetilde{V}_{n})^{\gamma })}{e^{\alpha
(U^{\delta }(x)\widetilde{V}_{n})^{\gamma }}}a_{n}\right] dx  \nonumber \\
&\geq &||f||_{p,\widetilde{\Phi }_{\delta }}J.  \label{4.10}
\end{eqnarray}%
Then by (\ref{4.8}), we have (\ref{4.6}).

On the other hand, assuming that (\ref{4.6}) is valid, we set $f(x)$ as in
Theorem 1. Then we find $J^{q}=||f||_{p,\widetilde{\Phi }_{\delta }}^{p}.$\\
If $J=\infty ,$ then (\ref{4.8}) is trivially valid. \\
If $J=0,$ then (\ref%
{4.8}) remains impossible.\\ 
Suppose that $0<J<\infty .$ By (\ref{4.6}), it
follows that%
\begin{eqnarray*}
||f||_{p,\widetilde{\Phi }_{\delta }}^{p} &=&J^{q}=I>k(\sigma )||f||_{p,%
\widetilde{\Phi }_{\delta }}||a||_{q,\widetilde{\Psi }}, \\
||f||_{p,\widetilde{\Phi }_{\delta }}^{p-1} &=&J>k(\sigma )||a||_{q,%
\widetilde{\Psi }},
\end{eqnarray*}%
and then (\ref{4.8}) follows, which is equivalent to (\ref{4.6}).

Therefore, inequalities (\ref{4.6}), (\ref{4.7}) and (\ref{4.8}) are
equivalent.

For $\varepsilon \in (0,p(\sigma -\gamma )),$ we set $\widetilde{\sigma }%
=\sigma +\frac{\varepsilon }{p}(>\gamma ),$ and $\widetilde{f}=\widetilde{f}%
(x),x\in \mathbf{R}_{+},\widetilde{a}=\{\widetilde{a}_{n}\}_{n=1}^{\infty },$%
\begin{eqnarray*}
\widetilde{f}(x) &=&\left\{
\begin{array}{c}
U^{\delta \widetilde{\sigma }-1}(x)\mu (x),0<x^{\delta }\leq 1 \\
0,x^{\delta }>0%
\end{array}%
\right. , \\
\widetilde{a}_{n} &=&\widetilde{V}_{n}^{\widetilde{\sigma }-\varepsilon
-1}\nu _{n}=\widetilde{V}_{n}^{\sigma -\frac{\varepsilon }{q}-1}\nu
_{n},n\in \mathbf{N}.
\end{eqnarray*}%
By (\ref{2.8}), (\ref{2.9}) and (\ref{3.12}), we obtain%
\begin{eqnarray*}
&&||\widetilde{f}||_{p,\widetilde{\Phi }_{\delta }}||\widetilde{a}||_{q,%
\widetilde{\Psi }} \\
&=&\left[ \int_{\{x>0;0<x^{\delta }\leq 1\}}(1-O((U(x))^{\delta (\sigma
-\gamma )}))\frac{\mu (x)dx}{U^{1-\delta \varepsilon }(x)}\right] ^{\frac{1}{%
p}}\left( \sum_{n=1}^{\infty }\frac{\nu _{n}}{\widetilde{V}%
_{n}^{1+\varepsilon }}\right) ^{\frac{1}{q}} \\
&=&\frac{1}{\varepsilon }\left( U^{\delta \varepsilon }(1)-\varepsilon
O(1)\right) ^{\frac{1}{p}}\left( \frac{1}{V_{n_{0}}^{\varepsilon }}%
+\varepsilon \widetilde{O}(1)\right) ^{\frac{1}{q}},
\end{eqnarray*}%
\begin{eqnarray*}
\widetilde{I} &=&\sum_{n=1}^{\infty }\int_{0}^{\infty }\frac{\csc h(\rho
(U^{\delta }(x)\widetilde{V}_{n})^{\gamma })}{e^{\alpha (U^{\delta }(x)%
\widetilde{V}_{n})^{\gamma }}}\widetilde{a}_{n}\widetilde{f}(x)dx \\
&=&\sum_{n=1}^{\infty }\left[ \int_{\{x>0;0<x^{\delta }\leq 1\}}\frac{\csc
h(\rho (U^{\delta }(x)\widetilde{V}_{n})^{\gamma })}{e^{\alpha (U^{\delta
}(x)\widetilde{V}_{n})^{\gamma }}}\frac{\widetilde{V}_{n}^{\widetilde{\sigma
}}\mu (x)}{U^{1-\delta \widetilde{\sigma }}(x)}dx\right] \frac{\nu _{n}}{%
\widetilde{V}_{n}^{1+\varepsilon }} \\
&\leq &\sum_{n=1}^{\infty }\left[ \int_{0}^{\infty }\frac{\csc h(\rho
(U^{\delta }(x)\widetilde{V}_{n})^{\gamma })}{e^{\alpha (U^{\delta }(x)%
\widetilde{V}_{n})^{\gamma }}}\frac{\widetilde{V}_{n}^{\widetilde{\sigma }%
}\mu (x)}{U^{1-\delta \widetilde{\sigma }}(x)}dx\right] \frac{\nu _{n}}{%
\widetilde{V}_{n}^{1+\varepsilon }} \\
&=&\sum_{n=1}^{\infty }\varpi _{\delta }(\widetilde{\sigma },n)\frac{\nu _{n}%
}{\widetilde{V}_{n}^{1+\varepsilon }}=k(\widetilde{\sigma }%
)\sum_{n=1}^{\infty }\frac{\nu _{n}}{\widetilde{V}_{n}^{1+\varepsilon }} \\
&=&\frac{1}{\varepsilon}k(\sigma +\frac{\varepsilon }{p})\left( \frac{1}{V_{n_{0}}^{\varepsilon }}%
+\varepsilon \widetilde{O}(1)\right) .
\end{eqnarray*}

If there exists a positive constant $K\geq k(\sigma ),$ such that (\ref{4.1}%
) is valid when replacing $k(\sigma )$ by $K,$ then in particular, we have $%
\varepsilon \widetilde{I}>\varepsilon K||\widetilde{f}||_{p,\widetilde{\Phi }%
_{\delta }}||\widetilde{a}||_{q,\widetilde{\Psi }},$ namely,%
\begin{eqnarray*}
&&k(\sigma +\frac{\varepsilon }{p})\left( \frac{1}{V_{n_{0}}^{\varepsilon }}%
+\varepsilon \widetilde{O}(1)\right) \\
&>&K\left( U^{\delta \varepsilon }(1)-\varepsilon O(1)\right) ^{\frac{1}{p}%
}\left( \frac{1}{V_{n_{0}}^{\varepsilon }}+\varepsilon \widetilde{O}%
(1)\right) ^{\frac{1}{q}}.
\end{eqnarray*}%
It follows that $k(\sigma )\geq K(\varepsilon \rightarrow 0^{+}).$ Hence, $%
K=k(\sigma )$ is the best possible constant factor of (\ref{4.6}).

The constant factor $k(\sigma )$ in (\ref{4.7}) ((\ref{4.8})) is still the
best possible. Otherwise, we would reach a contradiction by (\ref{4.9}) ((%
\ref{4.10})) that the constant factor in (\ref{4.6}) is not the best
possible. The theorem is proved.
\end{proof}

\section{Some Particular Inequalities}

For $\widetilde{\nu }_{n}=0,\widetilde{V}_{n}=V_{n},$ we set $$\Psi (n):=%
\frac{V_{n}^{q(1-\sigma )-1}}{\nu _{n}^{q-1}}\ (n\in \mathbf{N}).$$ In view of
Theorem 2-4, we have

\begin{corollary}
 If $0\leq \alpha \leq \rho (\rho >0),0<\gamma <\sigma
\leq 1,k(\sigma )$ is given by (\ref{2.1}), there exists $n_{0}\in
\mathbf{N},$ such that $v_{n}\geq v_{n+1}$ $(n\in \mathbf{\{}%
n_{0},n_{0}+1,\cdots \}),$ and $U(\infty )=V(\infty )=\infty ,$ then\\ 
(i) for
$p>1,$ $0<||f||_{p,\Phi _{\delta }},||a||_{q,\Psi }<\infty ,$ we have the
following equivalent inequalities:%
\begin{eqnarray}
&&\sum_{n=1}^{\infty }\int_{0}^{\infty }\frac{\csc h(\rho (U^{\delta
}(x)V_{n})^{\gamma })}{e^{\alpha (U^{\delta }(x)V_{n})^{\gamma }}}%
a_{n}f(x)dx <k(\sigma )||f||_{p,\Phi _{\delta }}||a||_{q,\Psi },  \label{5.1}
\\
&&\sum_{n=1}^{\infty }\frac{\nu _{n}}{V_{n}^{1-p\sigma }}\left[
\int_{0}^{\infty }\frac{\csc h(\rho (U^{\delta }(x)V_{n})^{\gamma })}{%
e^{\alpha (U^{\delta }(x)V_{n})^{\gamma }}}f(x)dx\right] ^{p} <k(\sigma
)||f||_{p,\Phi _{\delta }},  \label{5.2}
\end{eqnarray}%
\begin{equation}
\left\{ \int_{0}^{\infty }\frac{\mu (x)}{U^{1-q\delta \sigma }(x)}\left[
\sum_{n=1}^{\infty }\frac{\csc h(\rho (U^{\delta }(x)V_{n})^{\gamma })}{%
e^{\alpha (U^{\delta }(x)V_{n})^{\gamma }}}a_{n}\right] ^{q}dx\right\} ^{%
\frac{1}{q}}<k(\sigma )||a||_{q,\Psi };  \label{5.3}
\end{equation}%
(ii) for $p<0,$ $0<||f||_{p,\Phi _{\delta }},||a||_{q,\Psi }<\infty ,$ we
have the following equivalent inequalities:%
\begin{eqnarray}
&&\sum_{n=1}^{\infty }\int_{0}^{\infty }\frac{\csc h(\rho (U^{\delta
}(x)V_{n})^{\gamma })}{e^{\alpha (U^{\delta }(x)V_{n})^{\gamma }}}%
a_{n}f(x)dx >k(\sigma )||f||_{p,\Phi _{\delta }}||a||_{q,\Psi },  \label{5.4}
\\
&&\sum_{n=1}^{\infty }\frac{\nu _{n}}{V_{n}^{1-p\sigma }}\left[
\int_{0}^{\infty }\frac{\csc h(\rho (U^{\delta }(x)V_{n})^{\gamma })}{%
e^{\alpha (U^{\delta }(x)V_{n})^{\gamma }}}f(x)dx\right] ^{p} >k(\sigma
)||f||_{p,\Phi _{\delta }},  \label{5.5}
\end{eqnarray}%
\begin{equation}
\left\{ \int_{0}^{\infty }\frac{\mu (x)}{U^{1-q\delta \sigma }(x)}\left[
\sum_{n=1}^{\infty }\frac{\csc h(\rho (U^{\delta }(x)V_{n})^{\gamma })}{%
e^{\alpha (U^{\delta }(x)V_{n})^{\gamma }}}a_{n}\right] ^{q}dx\right\} ^{%
\frac{1}{q}}>k(\sigma )||a||_{q,\Psi };  \label{5.6}
\end{equation}%
(iii) for $0<p<1,$ $0<||f||_{p,\Phi _{\delta }},||a||_{q,\Psi }<\infty ,$ we
have the following equivalent inequalities:%
\begin{eqnarray}
&&\sum_{n=1}^{\infty }\int_{0}^{\infty }\frac{\csc h(\rho (U^{\delta
}(x)V_{n})^{\gamma })}{e^{\alpha (U^{\delta }(x)V_{n})^{\gamma }}}%
a_{n}f(x)dx >k(\sigma )||f||_{p,\widetilde{\Phi }_{\delta }}||a||_{q,\Psi },
\label{5.7} \\
&&\sum_{n=1}^{\infty }\frac{\nu _{n}}{V_{n}^{1-p\sigma }}\left[
\int_{0}^{\infty }\frac{\csc h(\rho (U^{\delta }(x)V_{n})^{\gamma })}{%
e^{\alpha (U^{\delta }(x)V_{n})^{\gamma }}}f(x)dx\right] ^{p} >k(\sigma
)||f||_{p,\widetilde{\Phi }_{\delta }},  \label{5.8}
\end{eqnarray}%
\begin{eqnarray}
&&\left\{ \int_{0}^{\infty }\frac{(1-\theta _{\delta }(\sigma ,x))^{1-q}\mu
(x)}{U^{1-q\delta \sigma }(x)}\left[ \sum_{n=1}^{\infty }\frac{\csc h(\rho
(U^{\delta }(x)V_{n})^{\gamma })}{e^{\alpha (U^{\delta }(x)V_{n})^{\gamma }}}%
a_{n}\right] ^{q}dx\right\} ^{\frac{1}{q}}  \nonumber \\
&>&k(\sigma )||a||_{q,\Psi }.  \label{5.9}
\end{eqnarray}
\end{corollary}

 The above inequalities have the best possible constant factor $%
k(\sigma ).$

In particular, for $\delta =1,$ we have the following inequalities with the
non-homogen-eous kernel:

\begin{corollary}
 If $0\leq \alpha \leq \rho (\rho >0),0<\gamma <\sigma
\leq 1,k(\sigma )$ is given by (\ref{2.1}), there exists $n_{0}\in
\mathbf{N},$ such that $v_{n}\geq v_{n+1}$ $(n\in \mathbf{\{}%
n_{0},n_{0}+1,\cdots \}),$ and $U(\infty )=V(\infty )=\infty ,$ then\\ 
(i) for
$p>1,$ $0<||f||_{p,\Phi _{1}},||a||_{q,\Psi }<\infty ,$ we have the
following equivalent inequalities:%
\begin{eqnarray}
&&\sum_{n=1}^{\infty }\int_{0}^{\infty }\frac{\csc h(\rho
(U(x)V_{n})^{\gamma })}{e^{\alpha (U(x)V_{n})^{\gamma }}}a_{n}f(x)dx
<k(\sigma )||f||_{p,\Phi _{1}}||a||_{q,\Psi },  \label{5.10} \\
&&\sum_{n=1}^{\infty }\frac{\nu _{n}}{V_{n}^{1-p\sigma }}\left[
\int_{0}^{\infty }\frac{\csc h(\rho (U(x)V_{n})^{\gamma })}{e^{\alpha
(U(x)V_{n})^{\gamma }}}f(x)dx\right] ^{p} <k(\sigma )||f||_{p,\Phi _{1}},
\label{5.11}
\end{eqnarray}%
\begin{equation}
\left\{ \int_{0}^{\infty }\frac{\mu (x)}{U^{1-q\sigma }(x)}\left[
\sum_{n=1}^{\infty }\frac{\csc h(\rho (U(x)V_{n})^{\gamma })}{e^{\alpha
(U(x)V_{n})^{\gamma }}}a_{n}\right] ^{q}dx\right\} ^{\frac{1}{q}}<k(\sigma
)||a||_{q,\Psi };  \label{5.12}
\end{equation}%
(ii) for $p<0,$ $0<||f||_{p,\Phi _{1}},||a||_{q,\Psi }<\infty ,$ we have the
following equivalent inequalities:%
\begin{eqnarray}
&&\sum_{n=1}^{\infty }\int_{0}^{\infty }\frac{\csc h(\rho
(U(x)V_{n})^{\gamma })}{e^{\alpha (U(x)V_{n})^{\gamma }}}a_{n}f(x)dx
>k(\sigma )||f||_{p,\Phi _{1}}||a||_{q,\Psi },  \label{5.13} \\
&&\sum_{n=1}^{\infty }\frac{\nu _{n}}{V_{n}^{1-p\sigma }}\left[
\int_{0}^{\infty }\frac{\csc h(\rho (U(x)V_{n})^{\gamma })}{e^{\alpha
(U(x)V_{n})^{\gamma }}}f(x)dx\right] ^{p} >k(\sigma )||f||_{p,\Phi _{1}},
\label{5.14}
\end{eqnarray}%
\begin{equation}
\left\{ \int_{0}^{\infty }\frac{\mu (x)}{U^{1-q\sigma }(x)}\left[
\sum_{n=1}^{\infty }\frac{\csc h(\rho (U(x)V_{n})^{\gamma })}{e^{\alpha
(U(x)V_{n})^{\gamma }}}a_{n}\right] ^{q}dx\right\} ^{\frac{1}{q}}>k(\sigma
)||a||_{q,\Psi };  \label{5.15}
\end{equation}%
(iii) for $0<p<1,$ $0<||f||_{p,\Phi _{1}},||a||_{q,\Psi }<\infty ,$ we have
the following equivalent inequalities:%
\begin{eqnarray}
&&\sum_{n=1}^{\infty }\int_{0}^{\infty }\frac{\csc h(\rho
(U(x)V_{n})^{\gamma })}{e^{\alpha (U(x)V_{n})^{\gamma }}}a_{n}f(x)dx
>k(\sigma )||f||_{p,\widetilde{\Phi }_{1}}||a||_{q,\Psi },  \label{5.16} \\
&&\sum_{n=1}^{\infty }\frac{\nu _{n}}{V_{n}^{1-p\sigma }}\left[
\int_{0}^{\infty }\frac{\csc h(\rho (U(x)V_{n})^{\gamma })}{e^{\alpha
(U(x)V_{n})^{\gamma }}}f(x)dx\right] ^{p} >k(\sigma )||f||_{p,\widetilde{%
\Phi }_{1}},  \label{5.17}
\end{eqnarray}%
\begin{eqnarray}
&&\left\{ \int_{0}^{\infty }\frac{(1-\theta _{1}(\sigma ,x))^{1-q}\mu (x)}{%
U^{1-q\sigma }(x)}\left[ \sum_{n=1}^{\infty }\frac{\csc h(\rho
(U(x)V_{n})^{\gamma })}{e^{\alpha (U(x)V_{n})^{\gamma }}}a_{n}\right]
^{q}dx\right\} ^{\frac{1}{q}}  \nonumber \\
&>&k(\sigma )||a||_{q,\Psi }.  \label{5.18}
\end{eqnarray}
\end{corollary}

The above inequalities involve the best possible constant factor $%
k(\sigma ).$

For $\delta =-1,$ we have the following inequalities with the homogeneous
kernel of degree 0:

\begin{corollary}
 If $0\leq \alpha \leq \rho (\rho >0),0<\gamma <\sigma
\leq 1,k(\sigma )$ is given by (\ref{2.1}), there exists $n_{0}\in
\mathbf{N},$ such that $v_{n}\geq v_{n+1}$ $(n\in \mathbf{\{}%
n_{0},n_{0}+1,\cdots \}),$ and $U(\infty )=V(\infty )=\infty ,$ then\\ 
(i) for
$p>1,$ $0<||f||_{p,\Phi _{-1}},||a||_{q,\Psi }<\infty ,$ we have the
following equivalent inequalities:%
\begin{eqnarray}
&&\sum_{n=1}^{\infty }\int_{0}^{\infty }\frac{\csc h(\rho (\frac{V_{n}}{U(x)}%
)^{\gamma })}{e^{\alpha (\frac{V_{n}}{U(x)})^{\gamma }}}a_{n}f(x)dx
<k(\sigma )||f||_{p,\Phi _{-1}}||a||_{q,\Psi },  \label{5.19} \\
&&\sum_{n=1}^{\infty }\frac{\nu _{n}}{V_{n}^{1-p\sigma }}\left[
\int_{0}^{\infty }\frac{\csc h(\rho (\frac{V_{n}}{U(x)})^{\gamma })}{%
e^{\alpha (\frac{V_{n}}{U(x)})^{\gamma }}}f(x)dx\right] ^{p} <k(\sigma
)||f||_{p,\Phi _{-1}},  \label{5.20}
\end{eqnarray}%
\begin{equation}
\left\{ \int_{0}^{\infty }\frac{\mu (x)}{U^{1+q\sigma }(x)}\left[
\sum_{n=1}^{\infty }\frac{\csc h(\rho (\frac{V_{n}}{U(x)})^{\gamma })}{%
e^{\alpha (\frac{V_{n}}{U(x)})^{\gamma }}}a_{n}\right] ^{q}dx\right\} ^{%
\frac{1}{q}}<k(\sigma )||a||_{q,\Psi };  \label{5.21}
\end{equation}%
(ii) for $p<0,$ $0<||f||_{p,\Phi _{-1}},||a||_{q,\Psi }<\infty ,$ we have
the following equivalent inequalities:%
\begin{eqnarray}
&&\sum_{n=1}^{\infty }\int_{0}^{\infty }\frac{\csc h(\rho (\frac{V_{n}}{U(x)}%
)^{\gamma })}{e^{\alpha (\frac{V_{n}}{U(x)})^{\gamma }}}a_{n}f(x)dx
>k(\sigma )||f||_{p,\Phi _{-1}}||a||_{q,\Psi },  \label{5.22} \\
&&\sum_{n=1}^{\infty }\frac{\nu _{n}}{V_{n}^{1-p\sigma }}\left[
\int_{0}^{\infty }\frac{\csc h(\rho (\frac{V_{n}}{U(x)})^{\gamma })}{%
e^{\alpha (\frac{V_{n}}{U(x)})^{\gamma }}}f(x)dx\right] ^{p} >k(\sigma
)||f||_{p,\Phi _{-1}},  \label{5.23}
\end{eqnarray}%
\begin{equation}
\left\{ \int_{0}^{\infty }\frac{\mu (x)}{U^{1+q\sigma }(x)}\left[
\sum_{n=1}^{\infty }\frac{\csc h(\rho (\frac{V_{n}}{U(x)})^{\gamma })}{%
e^{\alpha (\frac{V_{n}}{U(x)})^{\gamma }}}a_{n}\right] ^{q}dx\right\} ^{%
\frac{1}{q}}>k(\sigma )||a||_{q,\Psi };  \label{5.24}
\end{equation}%
(iii) for $0<p<1,$ $0<||f||_{p,\Phi _{-1}},||a||_{q,\Psi }<\infty ,$ we have
the following equivalent inequalities:%
\begin{eqnarray}
&&\sum_{n=1}^{\infty }\int_{0}^{\infty }\frac{\csc h(\rho (\frac{V_{n}}{U(x)}%
)^{\gamma })}{e^{\alpha (\frac{V_{n}}{U(x)})^{\gamma }}}a_{n}f(x)dx
>k(\sigma )||f||_{p,\widetilde{\Phi }_{-1}}||a||_{q,\Psi },  \label{5.25} \\
&&\sum_{n=1}^{\infty }\frac{\nu _{n}}{V_{n}^{1-p\sigma }}\left[
\int_{0}^{\infty }\frac{\csc h(\rho (\frac{V_{n}}{U(x)})^{\gamma })}{%
e^{\alpha (\frac{V_{n}}{U(x)})^{\gamma }}}f(x)dx\right] ^{p} >k(\sigma
)||f||_{p,\widetilde{\Phi }_{-1}},  \label{5.26}
\end{eqnarray}%
\begin{eqnarray}
&&\left\{ \int_{0}^{\infty }\frac{(1-\theta _{-1 }(\sigma ,x))^{1-q}\mu
(x)}{U^{1+q\sigma }(x)}\left[ \sum_{n=1}^{\infty }\frac{\csc h(\rho (\frac{%
V_{n}}{U(x)})^{\gamma })}{e^{\alpha (\frac{V_{n}}{U(x)})^{\gamma }}}a_{n}%
\right] ^{q}dx\right\} ^{\frac{1}{q}}  \nonumber \\
&>&k(\sigma )||a||_{q,\Psi }.  \label{5.27}
\end{eqnarray}
\end{corollary}

 The above inequalities involve the best possible constant factor $%
k(\sigma ).$

For $\alpha =\rho $ in Theorem 2-4, we have

\begin{corollary}
 If $\rho >0,0<\gamma <\sigma \leq 1,$%
\begin{equation}
K(\sigma ):=\frac{2\Gamma (\frac{\sigma }{\gamma })\zeta (\frac{\sigma }{%
\gamma })}{\gamma (2\rho )^{\sigma /\gamma }},  \label{5.28}
\end{equation}
there exists $n_{0}\in \mathbf{N},$ such that $v_{n}\geq v_{n+1}$ $(n\in \mathbf{\{}n_{0},n_{0}+1,\cdots \}),$ and $U(\infty
)=V(\infty )=\infty ,$ then\\ 
(i) for $p>1,$ $0<||f||_{p,\Phi _{\delta
}},||a||_{q,\widetilde{\Psi }}<\infty ,$ we have the following equivalent
inequalities with the best possible constant factor $K(\sigma )$:%
\begin{eqnarray}
&&\sum_{n=1}^{\infty }\int_{0}^{\infty }\frac{\csc h(\rho (U^{\delta }(x)%
\widetilde{V}_{n})^{\gamma })}{e^{\rho (U^{\delta }(x)\widetilde{V}%
_{n})^{\gamma }}}a_{n}f(x)dx <K(\sigma )||f||_{p,\Phi _{\delta }}||a||_{q,%
\widetilde{\Psi }},  \label{5.29} \\
&&\sum_{n=1}^{\infty }\frac{\nu _{n}}{\widetilde{V}_{n}^{1-p\sigma }}\left[
\int_{0}^{\infty }\frac{\csc h(\rho (U^{\delta }(x)\widetilde{V}%
_{n})^{\gamma })}{e^{\rho (U^{\delta }(x)\widetilde{V}_{n})^{\gamma }}}f(x)dx%
\right] ^{p} <K(\sigma )||f||_{p,\Phi _{\delta }},  \label{5.30}
\end{eqnarray}%
\begin{equation}
\left\{ \int_{0}^{\infty }\frac{\mu (x)}{U^{1-q\delta \sigma }(x)}\left[
\sum_{n=1}^{\infty }\frac{\csc h(\rho (U^{\delta }(x)\widetilde{V}%
_{n})^{\gamma })}{e^{\rho (U^{\delta }(x)\widetilde{V}_{n})^{\gamma }}}a_{n}%
\right] ^{q}dx\right\} ^{\frac{1}{q}}<K(\sigma )||a||_{q,\widetilde{\Psi }}.
\label{5.31}
\end{equation}%
(ii) for $p<0,0<||f||_{p,\Phi _{\delta }},||a||_{q,\widetilde{\Psi }}<\infty
,$ we have the following equivalent inequalities with the best possible
constant factor $K(\sigma ):$%
\begin{eqnarray}
&&\sum_{n=1}^{\infty }\int_{0}^{\infty }\frac{\csc h(\rho (U^{\delta }(x)%
\widetilde{V}_{n})^{\gamma })}{e^{\rho (U^{\delta }(x)\widetilde{V}%
_{n})^{\gamma }}}a_{n}f(x)dx >K(\sigma )||f||_{p,\Phi _{\delta }}||a||_{q,%
\widetilde{\Psi }},  \label{5.32} \\
&&\sum_{n=1}^{\infty }\frac{\nu _{n}}{\widetilde{V}_{n}^{1-p\sigma }}\left[
\int_{0}^{\infty }\frac{\csc h(\rho (U^{\delta }(x)\widetilde{V}%
_{n})^{\gamma })}{e^{\rho (U^{\delta }(x)\widetilde{V}_{n})^{\gamma }}}f(x)dx%
\right] ^{p} >K(\sigma )||f||_{p,\Phi _{\delta }},  \label{5.33}
\end{eqnarray}%
\begin{equation}
\left\{ \int_{0}^{\infty }\frac{\mu (x)}{U^{1-q\delta \sigma }(x)}\left[
\sum_{n=1}^{\infty }\frac{\csc h(\rho (U^{\delta }(x)\widetilde{V}%
_{n})^{\gamma })}{e^{\rho (U^{\delta }(x)\widetilde{V}_{n})^{\gamma }}}a_{n}%
\right] ^{q}dx\right\} ^{\frac{1}{q}}>K(\sigma )||a||_{q,\widetilde{\Psi }};
\label{5.34}
\end{equation}%
(iii) for $0<p<1,0<||f||_{p,\Phi _{\delta }},||a||_{q,\widetilde{\Psi }%
}<\infty ,$ we have the following equivalent inequalities with the best
possible constant factor $K(\sigma )$:%
\begin{eqnarray}
&&\sum_{n=1}^{\infty }\int_{0}^{\infty }\frac{\csc h(\rho (U^{\delta }(x)%
\widetilde{V}_{n})^{\gamma })}{e^{\rho (U^{\delta }(x)\widetilde{V}%
_{n})^{\gamma }}}a_{n}f(x)dx >K(\sigma )||f||_{p,\widetilde{\Phi }_{\delta
}}||a||_{q,\widetilde{\Psi }},  \label{5.35} \\
&&\sum_{n=1}^{\infty }\frac{\nu _{n}}{\widetilde{V}_{n}^{1-p\sigma }}\left[
\int_{0}^{\infty }\frac{\csc h(\rho (U^{\delta }(x)\widetilde{V}%
_{n})^{\gamma })}{e^{\rho (U^{\delta }(x)\widetilde{V}_{n})^{\gamma }}}f(x)dx%
\right] ^{p} >K(\sigma )||f||_{p,\widetilde{\Phi }_{\delta }},
\label{5.36}
\end{eqnarray}%
\begin{eqnarray}
&&\left\{ \int_{0}^{\infty }\frac{(1-\theta _{\delta }(\sigma ,x))^{1-q}\mu
(x)}{U^{1-q\delta \sigma }(x)}\left[ \sum_{n=1}^{\infty }\frac{\csc h(\rho
(U^{\delta }(x)\widetilde{V}_{n})^{\gamma })}{e^{\rho (U^{\delta }(x)%
\widetilde{V}_{n})^{\gamma }}}a_{n}\right] ^{q}dx\right\} ^{\frac{1}{q}}
\nonumber \\
&>&K(\sigma )||a||_{q,\widetilde{\Psi }}.  \label{5.37}
\end{eqnarray}
\end{corollary}

In particular, for $\gamma =\frac{\sigma }{2},\theta _{\delta }(\sigma
,x)=O((U(x))^{\frac{\delta \sigma }{2}}),$ \\
(i) for $p>1,$ we have the
following equivalent inequalities with the best possible constant factor $%
\frac{\pi ^{2}}{6\sigma \rho ^{2}}$:
\begin{eqnarray}
&&\sum_{n=1}^{\infty }\int_{0}^{\infty }\frac{\csc h(\rho (U^{\delta }(x)%
\widetilde{V}_{n})^{\sigma /2})}{e^{\rho (U^{\delta }(x)\widetilde{V}%
_{n})^{\sigma /2}}}a_{n}f(x)dx<\frac{\pi ^{2}}{6\sigma \rho ^{2}}%
||f||_{p,\Phi _{\delta }}||a||_{q,\widetilde{\Psi }},  \label{5.38} \\
&&\sum_{n=1}^{\infty }\frac{\nu _{n}}{\widetilde{V}_{n}^{1-p\sigma }}\left[
\int_{0}^{\infty }\frac{\csc h(\rho (U^{\delta }(x)\widetilde{V}%
_{n})^{\sigma /2})}{e^{\rho (U^{\delta }(x)\widetilde{V}_{n})^{\sigma /2}}}%
f(x)dx\right] ^{p}<\frac{\pi ^{2}}{6\sigma \rho ^{2}}||f||_{p,\Phi _{\delta
}},  \label{5.39}
\end{eqnarray}%
\begin{equation}
\left\{ \int_{0}^{\infty }\frac{\mu (x)}{U^{1-q\delta \sigma }(x)}\left[
\sum_{n=1}^{\infty }\frac{\csc h(\rho (U^{\delta }(x)\widetilde{V}%
_{n})^{\sigma /2})}{e^{\rho (U^{\delta }(x)\widetilde{V}_{n})^{\sigma /2}}}%
a_{n}\right] ^{q}dx\right\} ^{\frac{1}{q}}<\frac{\pi ^{2}}{6\sigma \rho ^{2}}%
||a||_{q,\widetilde{\Psi }};  \label{5.40}
\end{equation}%
(ii) for $p<0,$we have the following equivalent inequalities with the best
possible constant factor $\frac{\pi ^{2}}{6\sigma \rho ^{2}}$:%
\begin{eqnarray}
&&\sum_{n=1}^{\infty }\int_{0}^{\infty }\frac{\csc h(\rho (U^{\delta }(x)%
\widetilde{V}_{n})^{\sigma /2})}{e^{\rho (U^{\delta }(x)\widetilde{V}%
_{n})^{\sigma /2}}}a_{n}f(x)dx>\frac{\pi ^{2}}{6\sigma \rho ^{2}}%
||f||_{p,\Phi _{\delta }}||a||_{q,\widetilde{\Psi }},  \label{5.41} \\
&&\sum_{n=1}^{\infty }\frac{\nu _{n}}{\widetilde{V}_{n}^{1-p\sigma }}\left[
\int_{0}^{\infty }\frac{\csc h(\rho (U^{\delta }(x)\widetilde{V}%
_{n})^{\sigma /2})}{e^{\rho (U^{\delta }(x)\widetilde{V}_{n})^{\sigma /2}}}%
f(x)dx\right] ^{p}>\frac{\pi ^{2}}{6\sigma \rho ^{2}}||f||_{p,\Phi _{\delta
}},  \label{5.42}
\end{eqnarray}%
\begin{equation}
\left\{ \int_{0}^{\infty }\frac{\mu (x)}{U^{1-q\delta \sigma }(x)}\left[
\sum_{n=1}^{\infty }\frac{\csc h(\rho (U^{\delta }(x)\widetilde{V}%
_{n})^{\sigma /2})}{e^{\rho (U^{\delta }(x)\widetilde{V}_{n})^{\sigma /2}}}%
a_{n}\right] ^{q}dx\right\} ^{\frac{1}{q}}>\frac{\pi ^{2}}{6\sigma \rho ^{2}}%
||a||_{q,\widetilde{\Psi }};  \label{5.43}
\end{equation}%
(iii) for $0<p<1,$we have the following equivalent inequalities with the
best possible constant factor $\frac{\pi ^{2}}{6\sigma \rho ^{2}}$:%
\begin{eqnarray}
&&\sum_{n=1}^{\infty }\int_{0}^{\infty }\frac{\csc h(\rho (U^{\delta }(x)%
\widetilde{V}_{n})^{\sigma /2})}{e^{\rho (U^{\delta }(x)\widetilde{V}%
_{n})^{\sigma /2}}}a_{n}f(x)dx>\frac{\pi ^{2}}{6\sigma \rho ^{2}}||f||_{p,%
\widetilde{\Phi }_{\delta }}||a||_{q,\widetilde{\Psi }},  \label{5.44} \\
&&\sum_{n=1}^{\infty }\frac{\nu _{n}}{\widetilde{V}_{n}^{1-p\sigma }}\left[
\int_{0}^{\infty }\frac{\csc h(\rho (U^{\delta }(x)\widetilde{V}%
_{n})^{\sigma /2})}{e^{\rho (U^{\delta }(x)\widetilde{V}_{n})^{\sigma /2}}}%
f(x)dx\right] ^{p}>\frac{\pi ^{2}}{6\sigma \rho ^{2}}||f||_{p,\widetilde{%
\Phi }_{\delta }},  \label{5.45}
\end{eqnarray}%
\begin{eqnarray}
&&\left\{ \int_{0}^{\infty }\frac{(1-\theta _{\delta }(\sigma ,x))^{1-q}\mu
(x)}{U^{1-q\delta \sigma }(x)}\left[ \sum_{n=1}^{\infty }\frac{\csc h(\rho
(U^{\delta }(x)\widetilde{V}_{n})^{\sigma /2})}{e^{\rho (U^{\delta }(x)%
\widetilde{V}_{n})^{\sigma /2}}}a_{n}\right] ^{q}dx\right\} ^{\frac{1}{q}}
\nonumber \\
&>&\frac{\pi ^{2}}{6\sigma \rho ^{2}}||a||_{q,\widetilde{\Psi }}.
\label{5.46}
\end{eqnarray}

For $\alpha =0,\gamma =\frac{\sigma }{2},\theta _{\delta }(\sigma
,x)=O((U(x))^{\frac{\delta \sigma }{2}})$ in Theorem 2-4, we have

\begin{corollary}
 If $\rho >0,0<\sigma \leq 1,$ there exists $%
n_{0}\in \mathbf{N},$ such that $v_{n}\geq  v_{n+1}$ $(n\in
\mathbf{\{}n_{0},n_{0}+1,\cdots \}),$ and $U(\infty )=V(\infty )=\infty ,$
then\\ 
(i) for $p>1,$ $0<||f||_{p,\Phi _{\delta }},||a||_{q,\widetilde{\Psi }%
}<\infty ,$ we have the following equivalent inequalities with the best
possible constant factor $\frac{\pi ^{2}}{2\sigma \rho ^{2}}$:%
\begin{eqnarray}
&&\sum_{n=1}^{\infty }\int_{0}^{\infty }\csc h(\rho (U^{\delta }(x)%
\widetilde{V}_{n})^{\frac{\sigma }{2}})a_{n}f(x)dx<\frac{\pi ^{2}}{2\sigma
\rho ^{2}}||f||_{p,\Phi _{\delta }}||a||_{q,\widetilde{\Psi }},  \label{5.47}
\\
&&\sum_{n=1}^{\infty }\frac{\nu _{n}}{\widetilde{V}_{n}^{1-p\sigma }}\left[
\int_{0}^{\infty }\csc h(\rho (U^{\delta }(x)\widetilde{V}_{n})^{\frac{%
\sigma }{2}})f(x)dx\right] ^{p}<\frac{\pi ^{2}}{2\sigma \rho ^{2}}%
||f||_{p,\Phi _{\delta }},  \label{5.48}
\end{eqnarray}%
\begin{equation}
\left\{ \int_{0}^{\infty }\frac{\mu (x)}{U^{1-q\delta \sigma }(x)}\left[
\sum_{n=1}^{\infty }\csc h(\rho (U^{\delta }(x)\widetilde{V}_{n})^{\frac{%
\sigma }{2}})a_{n}\right] ^{q}dx\right\} ^{\frac{1}{q}}<\frac{\pi ^{2}}{%
2\sigma \rho ^{2}}||a||_{q,\widetilde{\Psi }};  \label{5.49}
\end{equation}%
(ii) for $p<0,0<||f||_{p,\Phi _{\delta }},||a||_{q,\widetilde{\Psi }}<\infty
,$ we have the following equivalent inequalities with the best possible
constant factor $\frac{\pi ^{2}}{2\sigma \rho ^{2}}$:%
\begin{eqnarray}
&&\sum_{n=1}^{\infty }\int_{0}^{\infty }\csc h(\rho (U^{\delta }(x)%
\widetilde{V}_{n})^{\frac{\sigma }{2}})a_{n}f(x)dx>\frac{\pi ^{2}}{2\sigma
\rho ^{2}}||f||_{p,\Phi _{\delta }}||a||_{q,\widetilde{\Psi }},  \label{5.50}
\\
&&\sum_{n=1}^{\infty }\frac{\nu _{n}}{\widetilde{V}_{n}^{1-p\sigma }}\left[
\int_{0}^{\infty }\csc h(\rho (U^{\delta }(x)\widetilde{V}_{n})^{\frac{%
\sigma }{2}})f(x)dx\right] ^{p}>\frac{\pi ^{2}}{2\sigma \rho ^{2}}%
||f||_{p,\Phi _{\delta }},  \label{5.51}
\end{eqnarray}%
\begin{equation}
\left\{ \int_{0}^{\infty }\frac{\mu (x)}{U^{1-q\delta \sigma }(x)}\left[
\sum_{n=1}^{\infty }\csc h(\rho (U^{\delta }(x)\widetilde{V}_{n})^{\frac{%
\sigma }{2}})a_{n}\right] ^{q}dx\right\} ^{\frac{1}{q}}>\frac{\pi ^{2}}{%
2\sigma \rho ^{2}}||a||_{q,\widetilde{\Psi }};  \label{5.52}
\end{equation}%
(iii) for $0<p<1,0<||f||_{p,\Phi _{\delta }},||a||_{q,\widetilde{\Psi }%
}<\infty ,$ we have the following equivalent inequalities with the best
possible constant factor $\frac{\pi ^{2}}{2\sigma \rho ^{2}}$:%
\begin{eqnarray}
&&\sum_{n=1}^{\infty }\int_{0}^{\infty }\csc h(\rho (U^{\delta }(x)%
\widetilde{V}_{n})^{\frac{\sigma }{2}})a_{n}f(x)dx>\frac{\pi ^{2}}{2\sigma
\rho ^{2}}||f||_{p,\widetilde{\Phi }_{\delta }}||a||_{q,\widetilde{\Psi }},
\label{5.53} \\
&&\sum_{n=1}^{\infty }\frac{\nu _{n}}{\widetilde{V}_{n}^{1-p\sigma }}\left[
\int_{0}^{\infty }\csc h(\rho (U^{\delta }(x)\widetilde{V}_{n})^{\frac{%
\sigma }{2}})f(x)dx\right] ^{p}>\frac{\pi ^{2}}{2\sigma \rho ^{2}}||f||_{p,%
\widetilde{\Phi }_{\delta }},  \label{5.54}
\end{eqnarray}%
\begin{eqnarray}
&&\left\{ \int_{0}^{\infty }\frac{(1-\theta _{\delta }(\sigma ,x))^{1-q}\mu
(x)}{U^{1-q\delta \sigma }(x)}\left[ \sum_{n=1}^{\infty }\csc h(\rho
(U^{\delta }(x)\widetilde{V}_{n})^{\frac{\sigma }{2}})a_{n}\right]
^{q}dx\right\} ^{\frac{1}{q}}  \nonumber \\
&>&\frac{\pi ^{2}}{2\sigma \rho ^{2}}||a||_{q,\widetilde{\Psi }}.
\label{5.55}
\end{eqnarray}
\end{corollary}

\begin{remark}
(i) For $\mu (x)=\nu _{n}=1$ in (\ref{5.1}), we
have the following inequality with the best possible constant factor $%
k(\sigma ):$%
\begin{eqnarray}
&&\sum_{n=1}^{\infty }\int_{0}^{\infty }\frac{\csc h(\rho (x^{\delta
}n)^{\gamma })}{e^{\alpha (x^{\delta }n)^{\gamma }}}a_{n}f(x)dx   \\
&<&k(\sigma )\left[ \int_{0}^{\infty }x^{p(1-\delta \sigma )-1}f^{p}(x)dx%
\right] ^{\frac{1}{p}}\left[ \sum_{n=1}^{\infty }n^{q(1-\sigma )-1}a_{n}^{q}%
\right] ^{\frac{1}{q}}.  \label{5.56}
\end{eqnarray}
\end{remark}

In particular, for $\delta =1,$ we have the following inequality with the
non-homogeneous kernel:%
\begin{eqnarray}
&&\sum_{n=1}^{\infty }\int_{0}^{\infty }\frac{\csc h(\rho (xn)^{\gamma })}{%
e^{\alpha (xn)^{\gamma }}}a_{n}f(x)dx   \\
&<&k(\sigma )\left[ \int_{0}^{\infty }x^{p(1-\sigma )-1}f^{p}(x)dx\right] ^{%
\frac{1}{p}}\left[ \sum_{n=1}^{\infty }n^{q(1-\sigma )-1}a_{n}^{q}\right] ^{%
\frac{1}{q}};  \label{5.57}
\end{eqnarray}%
for $\delta =-1,$ we have the following inequality with the homogeneous
kernel:%
\begin{eqnarray}
&&\sum_{n=1}^{\infty }\int_{0}^{\infty }\frac{\csc h(\rho (\frac{n}{x}%
)^{\gamma })}{e^{\alpha (\frac{n}{x})^{\gamma }}}a_{n}f(x)dx   \\
&<&k(\sigma )\left[ \int_{0}^{\infty }x^{p(1+\sigma )-1}f^{p}(x)dx\right] ^{%
\frac{1}{p}}\left[ \sum_{n=1}^{\infty }n^{q(1-\sigma )-1}a_{n}^{q}\right] ^{%
\frac{1}{q}}.  \label{5.58}
\end{eqnarray}

(ii) For $\mu (x)=\nu _{n}=1,$ $\widetilde{\nu }_{n}=\tau \in (0,%
\frac{1}{2}]$ in (\ref{3.1}), we have the following more accurate inequality
than (\ref{5.28}) with the best possible constant factor $k(\sigma ):$%
\begin{eqnarray}
&&\sum_{n=1}^{\infty }\int_{0}^{\infty }\frac{\csc h(\rho \lbrack x^{\delta
}(n-\tau )]^{\gamma })}{e^{\alpha (x^{\delta }(n-\tau )]^{\gamma })^{\gamma
}}}a_{n}f(x)dx   \\
&<&k(\sigma )\left[ \int_{0}^{\infty }x^{p(1-\delta \sigma )-1}f^{p}(x)dx%
\right] ^{\frac{1}{p}}\left[ \sum_{n=1}^{\infty }(n-\tau )^{q(1-\sigma
)-1}a_{n}^{q}\right] ^{\frac{1}{q}}.  \label{5.59}
\end{eqnarray}

In particular, for $\delta =1,$ we have the following inequality with the
non-homogeneous kernel:%
\begin{eqnarray}
&&\sum_{n=1}^{\infty }\int_{0}^{\infty }\frac{\csc h([x(n-\tau )]^{\gamma })%
}{e^{\alpha \{[x(n-\tau )]\}^{\gamma }}}a_{n}f(x)dx   \\
&<&k(\sigma )\left[ \int_{0}^{\infty }x^{p(1-\sigma )-1}f^{p}(x)dx\right] ^{%
\frac{1}{p}}\left[ \sum_{n=1}^{\infty }(n-\tau )^{q(1-\sigma )-1}a_{n}^{q}%
\right] ^{\frac{1}{q}};  \label{5.60}
\end{eqnarray}%
for $\delta =-1,$ we have the following inequality with the homogeneous
kernel:%
\begin{eqnarray}
&&\sum_{n=1}^{\infty }\int_{0}^{\infty }\frac{\csc h(\rho (\frac{n-\tau }{x}%
)^{\gamma })}{e^{\alpha (\frac{n-\tau }{x})^{\gamma }}}a_{n}f(x)dx   \\
&<&k(\sigma )\left[ \int_{0}^{\infty }x^{p(1+\sigma )-1}f^{p}(x)dx\right] ^{%
\frac{1}{p}}\left[ \sum_{n=1}^{\infty }(n-\tau )^{q(1-\sigma )-1}a_{n}^{q}%
\right] ^{\frac{1}{q}}.  \label{5.61}
\end{eqnarray}

We can still obtain a large number of other inequalities by using some
special parameters in the above Theorems and Corollaries.
$$\,\,$$
\textbf{Acknowledgements.}
B. Yang: This work is supported by the
National Natural Science Foundation of China (No. 61370186), and 2013
Knowledge Construction Special Foundation Item of Guangdong Institution of
Higher Learning College and University (No. 2013KJCX0140).

\end{document}